\documentclass[pdflatex,sn-mathphys-num]{sn-jnl}

\usepackage{graphicx}
\usepackage{multirow}
\usepackage{booktabs}
\usepackage{rotating}
\usepackage{amsmath,amssymb,amsfonts}
\usepackage{amsthm}
\usepackage{mathrsfs}
\usepackage[title]{appendix}
\usepackage{xcolor}
\usepackage{textcomp}
\usepackage{manyfoot}
\usepackage{algorithm}
\usepackage{algorithmicx}
\usepackage{algpseudocode}
\usepackage{listings}
\usepackage{float}
\usepackage[section]{placeins}
\usepackage{subcaption}

\theoremstyle{thmstyleone}
\newtheorem{theorem}{Theorem}[section]

\newtheorem{lemma}{Lemma}[section]
\newtheorem{corollary}{Corollary}[section]

\theoremstyle{thmstylethree}
\newtheorem{example}{Example}[section]
\newtheorem{remark}{Remark}[section]
\newtheorem{definition}{Definition}[section]
\newtheorem{method}{Method}[section]
\algrenewcommand\algorithmicrequire{\textbf{Input:}}
\algrenewcommand\algorithmicensure{\textbf{Output:}}

\begin{document}

\title[Article Title]{Solving Large-Scale Sparse Linear Complementarity Problems Using a Relaxed Shifted Matrix-Splitting Modulus-Based Iteration Method}


\author*[1]{\fnm{Abhishek Kumar Singh} \sur{Sengar}}\email{abhishek47sengar@gmail.com}

\author[2]{\fnm{Bharat} \sur{Kumar}}\email{bharatnishad.kanpu@gmail.com}

\author[1]{\fnm{Deepmala}}\email{dmrai23@gmail.com}

\affil[1]{\orgdiv{Discipline of Natural Sciences }, \orgname{Indian Institute of Information Technology Design and Manufacturing}, \orgaddress{ \city{Jabalpur}, \postcode{482005}, \state{Madhya Pradesh}, \country{India}}}

\affil[2]{\orgdiv{Department of Economic Sciences}, \orgname{Indian Institute of Technology}, \orgaddress{\city{Kanpur}, \postcode{208016}, \state{Uttar Pradesh}, \country{India}}}



\abstract{
	A relaxed shifted matrix-splitting modulus-based iteration method is proposed
	for solving large-scale sparse linear complementarity problems. The method
	extends the classical modulus-based matrix-splitting framework by introducing
	two shift parameters into the matrix splitting while preserving the original
	coefficient matrix. Different choices of the splitting lead to shifted
	Jacobi, Gauss-Seidel, SOR, and AOR-type schemes. Sufficient convergence conditions are established for \(P\)-matrices and \(H_+\)-matrices. For structured sparse matrices, practical AOR-type conditions and explicit admissible intervals for the shift parameters are also
	derived. These intervals are used together with a grid-search procedure to
	select efficient parameter pairs. Numerical experiments on four large-scale
	sparse test problems, including a quasi-variational inequality model, show that
	the proposed method is competitive with several existing modulus-based and
	projected methods. A direct comparison with the two-parameter MAOR method
	further illustrates the effect of the proposed shifted splitting.
}

\keywords{Shifted splitting, Modulus-based method, Linear complementarity problems, Large sparse systems.}


\pacs[MSC Classification]{90C33, 65F10, 65F50}

\maketitle

\section{Introduction}\label{sec1}

Let $A\in\mathbb{R}^{n\times n}$ and $q\in\mathbb{R}^n$ be given. The linear
complementarity problem, denoted by $\operatorname{LCP}(q,A)$, is to determine
vectors $z,w\in\mathbb{R}^n$ such that
\begin{equation}
	w=Az+q,\qquad z\geq 0,\qquad w\geq 0,\qquad z^T w=0.
	\label{eq1}
\end{equation}
The complementarity condition \(z^T w=0\) in \eqref{eq1} implies that, for
each component, at least one of \(z_i\) and \(w_i\) must be zero. This condition
makes the problem nonlinear in nature, even though the mapping \(Az+q\) is
linear.

Linear complementarity problems arise naturally in mathematical models where
nonnegativity and complementarity conditions must be satisfied simultaneously.
They provide a unified framework for problems in optimization, equilibrium
theory, variational inequalities, free-boundary problems, and engineering
applications \cite{Murty1988,CottlePangStone2009}. In many practical
situations, the LCP is not obtained directly, but appears after discretizing a
continuous model.

One important example comes from financial mathematics. The pricing of
American options leads to a free-boundary-type problem, and after suitable
discretization it can be reformulated as a linear complementarity problem.
Fixed-point methods for such LCP formulations have been studied in
\cite{ShiYangHuang2016}. This illustrates the practical relevance of developing
efficient numerical methods for LCPs, especially when the resulting matrix is
large and sparse.

Many numerical methods have been developed for solving
\(\operatorname{LCP}(q,A)\) in \eqref{eq1}. Classical approaches include
pivotal methods, interior-point methods, and iterative methods
\cite{Lemke1965,CottleDantzig1968,KojimaMegiddoYe1992}. Pivotal methods have
a strong theoretical foundation and are useful for many finite-dimensional
complementarity problems, while interior-point methods provide another powerful
class of solvers. However, when the coefficient matrix is large and sparse,
iterative methods are often more attractive because they require less storage
and can exploit the structure of the matrix.

Among iterative methods, projected relaxation and projected splitting methods
form an important class. These methods combine relaxation or matrix-splitting
ideas with a projection onto the nonnegative orthant. Early contributions in
this direction include the works of Cryer, Mangasarian, Ahn, and Pang
\cite{Cryer1971,Mangasarian1977,Ahn1981,Pang1984}. Later, projected Jacobi,
Gauss-Seidel, SOR, and AOR-type methods were developed and analyzed under
different matrix assumptions. Further modifications and parameter-dependent
variants were also studied to improve the convergence behavior of these methods
\cite{YuanSong2003,LiDai2007,DehghanHajarian2009}.

Another widely used approach is based on the modulus transformation. The main
idea is to replace the complementarity conditions by an equivalent fixed-point
equation involving the absolute value of an auxiliary variable. This avoids the
explicit projection step and allows matrix-splitting techniques to be applied.
The modulus idea goes back to van Bokhoven \cite{vanBokhoven1981} and was
further developed in several forms. Dong and Jiang proposed a modified modulus
method \cite{DongJiang2009}, and Hadjidimos and Tzoumas studied nonstationary
extrapolated modulus algorithms \cite{HadjidimosTzoumas2009}. A major
development was given by Bai, who introduced a general modulus-based matrix
splitting framework for LCPs, including modulus-based Jacobi, Gauss-Seidel,
SOR, and AOR-type schemes \cite{Bai2010}.

Since Bai's work, many extensions of modulus-based matrix splitting methods
have appeared. Zhang proposed two-step modulus-based matrix splitting methods
\cite{Zhang2011}, Zheng and Yin introduced accelerated modulus-based methods
\cite{ZhengYin2013}, and Bai and Zhang developed synchronous multisplitting
versions \cite{BaiZhang2013a,BaiZhang2013b}. Li, Xu, and Li-Zheng studied
general, modified, and preconditioned modulus-based schemes
\cite{Li2013,Xu2015,LiZheng2016}. Wu and Li proposed two-sweep and new
modulus-based matrix splitting methods \cite{WuLi2016,Wu2022}, and related
improvements were further studied in \cite{KumarDeepmalaDuttaDas2023}. Fang
and Zhu developed a modulus-based matrix double splitting method
\cite{Fang2019}. More recently, relaxation, preconditioned, double-relaxation, improved, and
momentum-accelerated variants have also been investigated for different matrix
classes
\cite{ZhengLiVong2017,WenZhengLiPeng2018,RenWangTangWang2019,HuangCui2023,LiuWangHuang2025,LiuZhang2026}.

Despite these developments, the performance of splitting-type methods still
depends strongly on two factors: the chosen matrix splitting and the iteration
parameters \cite{Cvetkovic2014,Xu2015,Fang2019}. For large sparse problems, reducing the iteration count alone may not reduce the total CPU time, because each iteration requires the solution of a linear system. Thus, the coefficient matrix in this linear system should retain a sparse and easily solvable structure. Hence, it is useful to introduce additional parameters in such a way that the original LCP is unchanged, while the resulting iteration can be tuned for better convergence behavior and lower computational cost.

Shifted-splitting ideas have been used in the construction of efficient
preconditioners and iterative solvers for other classes of matrix problems, such
as non-Hermitian positive definite systems and saddle point problems
\cite{BaiYinSu2006,CaoEtAl2014,LiMa2019}. These works suggest
that shifted forms of matrix splittings can improve the practical behavior of
iterative methods. Inspired by this idea, we introduce shifted terms into the
modulus-based matrix splitting framework for large sparse
\(\operatorname{LCP}(q,A)\).

In this paper, we propose a relaxed shifted matrix-splitting modulus-based
iteration framework. The method starts from Bai's modulus-based matrix
splitting formulation \cite{Bai2010} and replaces the splitting \(A=M-N\) by the shifted
splitting
$
A=(M+\Omega_1-\Omega_2)-(N+\Omega_1-\Omega_2).
$
The additional matrices \(\Omega_1\) and \(\Omega_2\) introduce tunable real
shift parameters \(\phi_1\) and \(\phi_2\) without changing the original
coefficient matrix of the LCP.

The shifted terms affect both the coefficient matrix of the linear system and
the linear part of the right-hand side of the iteration. Therefore, the proposed
RSMGS scheme is algebraically different from simply applying the MMS method
with an AOR splitting, namely the MAOR method of Bai \cite{Bai2010}.

The main contributions of this paper are summarized as follows:
\begin{enumerate}
	\item A relaxed shifted matrix-splitting modulus-based framework is
	presented by introducing two shift matrices \(\Omega_1\) and
	\(\Omega_2\) into Bai's MMS method, while the original coefficient
	matrix of the LCP is kept unchanged. This shows how shifted splitting can
	be used to improve the convergence performance of modulus-based
	iterations.
	
	\item Convergence results are established for the proposed method under
	\(P\)-matrix and \(H_+\)-matrix assumptions, and practical AOR-type
	sufficient conditions are obtained for structured sparse matrices.
	
	\item Explicit admissible intervals are derived for the shift parameters
	\((\phi_1,\phi_2)\), and the experimentally selected parameters are
	verified to satisfy the sufficient convergence condition for the tested
	problems.
	
	\item A grid-search procedure is developed for selecting efficient
	parameter pairs, and the sensitivity of the method with respect to
	different choices of the shift matrices \(\Omega_1\) and \(\Omega_2\) is
	investigated.
	
\item The proposed RSMGS method is compared with several existing
modulus-based and projected methods. A direct comparison with the
two-parameter MAOR method is also included to show that the shifted
splitting changes the iteration mechanism and can lead to better
computational performance.
\end{enumerate}

The rest of the paper is organized as follows. Section~\ref{sec2} recalls the
basic notation, definitions, and auxiliary results used in the analysis.
Section~\ref{sec3} reviews Bai's MMS method and presents the proposed relaxed
shifted matrix-splitting modulus-based iteration framework. Section~\ref{sec4}
establishes the convergence results and derives admissible intervals for the
parameters. Section~\ref{sec5} reports the numerical experiments, including
parameter selection, sensitivity analysis, comparison with existing methods,
and comparison with MAOR. Finally, Section~\ref{sec6} gives the concluding
remarks.

\section{Preliminaries}\label{sec2}
In this section, we recall the notation and auxiliary results used in the
convergence analysis. For matrices and vectors of the same size, inequalities
are understood componentwise; in particular, \(X\geq Y\) means that every
entry of \(X-Y\) is nonnegative, and \(X>Y\) means that every entry of
\(X-Y\) is positive. The notation \(|X|\) denotes the matrix or vector
obtained by taking absolute values componentwise. The transpose and inverse
of a matrix \(X\) are denoted by \(X^T\) and \(X^{-1}\), respectively,
whenever the inverse exists. The spectral radius of \(X\) is denoted by
\(\rho(X)\). For a vector \(x\) $\in \mathbb{R}^{n}$, \(\|x\|_2\) denotes the Euclidean norm, the notation \(|x|\) denotes the vector whose entries are \(|x_i|\), \(i=1,2,\ldots,n\), and
\(\min(x,y)\) is understood componentwise.

Throughout the paper, when the coefficient matrix is written as
\(A=D-L-U\), the matrices \(D\), \(L\), and \(U\) denote the diagonal,
strictly lower triangular, and strictly upper triangular parts of \(A\),
respectively.

The following definitions are standard in matrix theory; see \cite{BermanPlemmons1994,Varga2000}.

\begin{definition}
	Let \(A=(a_{ij})\in\mathbb{R}^{n\times n}\). The comparison matrix of \(A\),
	denoted by \(\langle A\rangle\), is defined by
	$
	\langle A\rangle_{ii}=|a_{ii}|,~
	\langle A\rangle_{ij}=-|a_{ij}|,~ i\neq j.
	$
\end{definition}

\begin{definition}
	A matrix \(A\in\mathbb{R}^{n\times n}\) is called a \(Z\)-matrix if all its
	off-diagonal entries are non-positive. A nonsingular \(Z\)-matrix \(A\) is
	called a nonsingular \(M\)-matrix if \(A^{-1}\geq0\).
\end{definition}

\begin{definition}
	A matrix \(A\in\mathbb{R}^{n\times n}\) is called an \(H\)-matrix if its
	comparison matrix \(\langle A\rangle\) is a nonsingular \(M\)-matrix. If,
	in addition, the diagonal entries of \(A\) are positive, then \(A\) is called
	an \(H_+\)-matrix.
\end{definition}

\begin{definition}
	A matrix \(A\in\mathbb{R}^{n\times n}\) is called a \(P\)-matrix if all its
	principal minors are positive.
\end{definition}

\begin{definition}
	Let \(A=B-C\) be a splitting of \(A\in\mathbb{R}^{n\times n}\). The splitting
	is called regular if \(B^{-1}\geq0\) and \(C\geq0\). It is called
	\(H\)-compatible if
	$
	\langle A\rangle=\langle B\rangle-|C|.
	$
\end{definition}

The following results will be used later.

\begin{lemma}[\cite{Horn2013}]
	\label{lem2}
	Let $M\in\mathbb{R}^{n\times n}$. Then $\rho(M)<1$ if and only if $\lim_{k\to\infty}M^k=0.$
\end{lemma}

\begin{lemma}[\cite{BermanPlemmons1994}]
	\label{lem:inverse-comparison}
	Let $F\in\mathbb{R}^{n\times n}$ be an $H$-matrix. Then $F$ is nonsingular and $|F^{-1}|\leq \langle F\rangle^{-1}$,	where $\langle F\rangle$ denotes the comparison matrix of $F$.
\end{lemma}

\begin{lemma}[\cite{BermanPlemmons1994}]
	\label{lem:regular-splitting}
	Let $E=F-G$ be a regular splitting, that is,
	$F^{-1}\geq0$ and $G\geq0$. If $E$ is a nonsingular $M$-matrix, then $\rho(F^{-1}G)<1$.
\end{lemma}

\begin{lemma}[\cite{BermanPlemmons1994}]
	\label{lem:m-matrix-characterization}
	Let $H\in\mathbb{R}^{n\times n}$ be nonnegative and let $s>0$. Then
	$sI-H$ is a nonsingular $M$-matrix if and only if $s>\rho(H)$.
\end{lemma}

\begin{lemma}[\cite{BermanPlemmons1994}]
	\label{lem:spectral-monotonicity}
	If $X,Y\in\mathbb{R}^{n\times n}$ satisfy
	$0\leq X\leq Y$, then $\rho(X)\leq\rho(Y)$.
\end{lemma}

\begin{lemma}[\cite{BermanPlemmons1994}]
	\label{lem:M-semipositive}
	Let $A\in\mathbb{R}^{n\times n}$ be a $Z$-matrix. Then $A$ is a nonsingular
	$M$-matrix if and only if there exists a vector $w>0$ such that
$Aw>0.$
\end{lemma}

\begin{lemma}[\cite{Bai2010}]
	\label{lem1}
	Let \(A=M-N\) be a splitting of \(A\in\mathbb{R}^{n\times n}\). Let
	\(\Omega\) be a positive diagonal matrix and let \(\gamma>0\). Then
	\(\operatorname{LCP}(q,A)\) defined in \eqref{eq1} is equivalent to the
	following fixed-point equation:
	\begin{equation}
		(M+\gamma\Omega)x
		=
		Nx+(\gamma\Omega-A)|x|-\gamma q.
		\label{eq:basic-modulus}
	\end{equation}
	More precisely, if \((z,w)\) is a solution of \eqref{eq1}, then
	\begin{equation*}
		x=\frac{1}{2}\left(\gamma z-\Omega^{-1}w\right)
	\end{equation*}
	satisfies \eqref{eq:basic-modulus}. Conversely, if \(x\) satisfies
	\eqref{eq:basic-modulus}, then
	\begin{equation}
		z=\frac{1}{\gamma}(|x|+x),
		\qquad
		w=\Omega(|x|-x)
		\label{eq:zw-basic}
	\end{equation}
	is a solution of \eqref{eq1}.
\end{lemma}

\section{Relaxed Shifted Matrix-Splitting Modulus-Based Iteration Method}\label{sec3}

The fixed-point formulation in Lemma~\ref{lem1} naturally leads to the
modulus-based matrix-splitting iteration of Bai \cite{Bai2010}, stated below.

\begin{method}[\cite{Bai2010} MMS method]
	\label{meth:bai-mms}\leavevmode\\
	Let \(A=M-N\) be a splitting of \(A\), where \(M\) is nonsingular. Let
	\(\Omega\) be a positive diagonal matrix and let \(\gamma>0\). Starting from
	an initial vector \(x^{(0)}\in\mathbb{R}^n\), compute \(x^{(k+1)}\) from
	\begin{equation}
		(M+\gamma\Omega)x^{(k+1)}
		=
		Nx^{(k)}
		+
		(\gamma\Omega-A)|x^{(k)}|
		-
		\gamma q.
		\label{eq:bai-mms}
	\end{equation}
	The corresponding approximation to the LCP solution is given by
	\begin{equation}
		z^{(k+1)}
		=
		\frac{1}{\gamma}
		\left(|x^{(k+1)}|+x^{(k+1)}\right).
		\label{eq:bai-z-recovery}
	\end{equation}
\end{method}
The splitting \(A=M-N\) can be written in the shifted form
\begin{equation}
	A
	=
	M-N
	=
	(M+\Omega_1-\Omega_2)
	-
	(N+\Omega_1-\Omega_2),
	\label{eq:shifted-splitting}
\end{equation}
where \(\Omega_1,\Omega_2\in\mathbb{R}^{n\times n}\) are shift parameter
matrices introduced to relax the underlying matrix splitting. The two added terms cancel each other, and therefore the coefficient matrix
\(A\) remains unchanged. However, the matrices used in the iteration are now
\(M+\Omega_1-\Omega_2\) and \(N+\Omega_1-\Omega_2\), instead of \(M\) and
\(N\). Thus, the shifted splitting modifies the iterative process while
preserving the original LCP. This leads to the following relaxed shifted
matrix-splitting modulus-based iteration method.

\begin{method}[Relaxed shifted matrix-splitting modulus-based iteration method]
	\label{meth:rsmms}\leavevmode\\
	For \(\gamma>0\), a positive diagonal matrix \(\Omega_4\), and an initial
	vector \(x^{(0)}\in\mathbb{R}^n\), compute \(x^{(k+1)}\) from
	\begin{equation}
		\left(M+\Omega_1-\Omega_2+\gamma\Omega_4\right)x^{(k+1)}
		=
		\left(N+\Omega_1-\Omega_2\right)x^{(k)}
		+
		\left(\gamma\Omega_4-A\right)|x^{(k)}|
		-\gamma q.
		\label{eq:proposed-method}
	\end{equation}
	The corresponding approximation to the LCP solution is obtained by
	\begin{equation}
		z^{(k+1)}
		=
		\Omega_3
		\left(|x^{(k+1)}|+x^{(k+1)}\right),
		\qquad
		\Omega_3=\frac{1}{\gamma}I.
		\label{eq:z-update}
	\end{equation}
\end{method}

Following the standard residual measure used in \cite{Bai2010,Wu2022}, the
residual at the \((k+1)\) iteration is defined by
\begin{equation}
	\operatorname{RES}(z^{(k+1)})
	=
	\left\|
	\min\left(Az^{(k+1)}+q,z^{(k+1)}\right)
	\right\|_2,
	\label{eq:residual}
\end{equation}
where the minimum is taken componentwise. The iteration is stopped when
\(\operatorname{RES}(z^{(k+1)})<\varepsilon\). The complete computational
procedure is given in Algorithm~\ref{algo1}.

\begin{algorithm}[h!]
	\caption{The relaxed shifted matrix-splitting modulus-based iteration method}
	\label{algo1}
	\begin{algorithmic}[1]
		
		\Require Matrix \(A\), vector \(q\), shifted splitting
		\(A=(M+\Omega_1-\Omega_2)-(N+\Omega_1-\Omega_2)\), matrices
		\(\Omega_3,\Omega_4\), parameter \(\gamma>0\), initial vector
		\(x^{(0)}\), tolerance \(\varepsilon>0\), and maximum iteration number
		\(K_{\max}\).
		
		\Ensure Approximate solution \(z^\star\), iteration count \(IT\), and
		residual \(RES^\star\).
		
		\For{\(k=0,1,2,\ldots,K_{\max}-1\)}
		
		\State Compute \(x^{(k+1)}\) from
		\begin{equation*}
			\left(M+\Omega_1-\Omega_2+\gamma\Omega_4\right)x^{(k+1)}
			=
			\left(N+\Omega_1-\Omega_2\right)x^{(k)}
			+
			\left(\gamma\Omega_4-A\right)|x^{(k)}|
			-\gamma q.
		\end{equation*}
		
		\State Update
		\begin{equation*}
			z^{(k+1)}
			=
			\Omega_3\left(|x^{(k+1)}|+x^{(k+1)}\right).
		\end{equation*}
		
		\State Compute
		\begin{equation*}
			RES^{(k+1)}
			=
			\left\|
			\min\left(Az^{(k+1)}+q,z^{(k+1)}\right)
			\right\|_2.
		\end{equation*}
		
		\If{\(RES^{(k+1)}<\varepsilon\)}
		\State Set \(z^\star=z^{(k+1)}\), \(IT=k+1\), and
		\(RES^\star=RES^{(k+1)}\).
		\State \Return \(z^\star\), \(IT\), and \(RES^\star\).
		\EndIf
		
		\EndFor
		
		\State \Return ``The method does not converge within \(K_{\max}\) iterations."
		
	\end{algorithmic}
\end{algorithm}

The shift parameter matrices \(\Omega_1\) and \(\Omega_2\) can be chosen from
simple sparse components of \(A\), such as \(I\), \(D\), \(L\), or \(U\), so
that the sparsity of the coefficient matrices is preserved. For example, one
may take \(\Omega_1=\phi_1I\) and \(\Omega_2=\phi_2L\),
\(\Omega_1=\phi_1I\) and \(\Omega_2=\phi_2U\), or
\(\Omega_1=\phi_1I\) and \(\Omega_2=\phi_2D\), where
\(\phi_1,\phi_2\in\mathbb{R}\) are relaxation parameters.

Let \(\alpha>0\), \(\beta\in\mathbb{R}\), and choose
\begin{equation}
	M
	=
	\frac{1}{\alpha}(D-\beta L),
	\qquad
	N
	=
	\frac{1}{\alpha}\left[(1-\alpha)D+(\alpha-\beta)L+\alpha U\right].
	\label{eq:aor-splitting}
\end{equation}
Substituting this splitting into Method~\ref{meth:rsmms} gives the relaxed
shifted matrix-splitting modulus-based accelerated overrelaxation (RSMAOR) iterative method.
In particular, \((\alpha,\beta)=(1,0)\) gives the relaxed shifted
matrix-splitting modulus-based Jacobi (RSMJ) iterative method, while
\((\alpha,\beta)=(1,1)\) gives the relaxed shifted matrix-splitting
modulus-based Gauss-Seidel (RSMGS) iterative method . When \(\beta=\alpha\neq0\), the method
reduces to the relaxed shifted matrix-splitting modulus-based SOR (RSMSOR) iterative method.

\section{Convergence Analysis}\label{sec4}

In this section, we study the convergence of the proposed shifted-splitting
multi-parameter modulus-based iteration method. We first establish a general
spectral-radius condition for the case where $A$ is a $P$-matrix. Then we derive
more explicit and verifiable sufficient conditions for important special choices
of the splitting and the parameter matrices. Finally, we discuss convergence for
the class of $H_{+}$-matrices.

\begin{theorem}
	\label{thm:p-matrix-general}
	Let $A\in\mathbb{R}^{n\times n}$ be a $P$-matrix and let $A=M-N$ be a
	splitting of $A$. Let $\gamma>0$, and let
	$\Omega_4$ be a positive diagonal matrix. Define
	$
	B=M+\Omega_1-\Omega_2+\gamma\Omega_4,
	$
	$
	C=N+\Omega_1-\Omega_2
	$
	and
	$
	G=\gamma\Omega_4-A.
	$
	Assume that $B$ is nonsingular. If
	\begin{equation}
		\rho\left(|B^{-1}C|+|B^{-1}G|\right)<1,
		\label{eq:p-matrix-general-condition}
	\end{equation}
	then, for any initial vector $x^{(0)}\in\mathbb{R}^{n}$, the sequence
	$\{x^{(k)}\}_{k=0}^{\infty}$ generated by \eqref{eq:proposed-method}
	converges to a vector $x^{\ast}$. Consequently,
	$z^{(k)}=\Omega_3(|x^{(k)}|+x^{(k)})$ converges to the unique solution
	$z^{\ast}$ of $\operatorname{LCP}(q,A)$.
\end{theorem}

\begin{proof}
	Since $A$ is a $P$-matrix, the problem $\operatorname{LCP}(q,A)$ has a unique
	solution $z^{\ast}$. By Lemma~\ref{lem1}, there exists a corresponding
	modulus vector $x^{\ast}$ satisfying
	\begin{equation}
		Bx^{\ast}
		=
		Cx^{\ast}+G|x^{\ast}|-\gamma q .
		\label{eq:fixed-point-star}
	\end{equation}
	Subtracting \eqref{eq:fixed-point-star} from the iteration
	\eqref{eq:proposed-method}, we obtain
	\begin{equation}
		x^{(k+1)}-x^{\ast}
		=
		B^{-1}C(x^{(k)}-x^{\ast})
		+
		B^{-1}G\left(|x^{(k)}|-|x^{\ast}|\right).
		\label{eq:error-equation}
	\end{equation}
	Using the componentwise inequality
	\begin{equation*}
		\left||x|-|y|\right|\leq |x-y|,
	\end{equation*}
	it follows from \eqref{eq:error-equation} that
	\begin{equation}
		|x^{(k+1)}-x^{\ast}|
		\leq
		\left(|B^{-1}C|+|B^{-1}G|\right)
		|x^{(k)}-x^{\ast}|.
		\label{eq:sharp-error-bound}
	\end{equation}
	Let
	\begin{equation*}
		T=|B^{-1}C|+|B^{-1}G|.
	\end{equation*}
	Then \eqref{eq:sharp-error-bound} implies
	\begin{equation*}
		|x^{(k)}-x^{\ast}|
		\leq
		T^k |x^{(0)}-x^{\ast}|,\qquad k=0,1,2,\ldots .
	\end{equation*}
	Since $\rho(T)<1$, Lemma~\ref{lem2} gives $T^k\to0$ as $k\to\infty$.
	Hence $x^{(k)}\to x^{\ast}$. Finally, by Lemma~\ref{lem1}, the vector
	$z^{(k)}=\Omega_3(|x^{(k)}|+x^{(k)})$ corresponds to the LCP variable.
	Since this mapping is continuous, we obtain $z^{(k)}\to z^{\ast}$, where
	$z^{\ast}$ is the unique solution of $\operatorname{LCP}(q,A)$.
\end{proof}

\begin{remark}
	The convergence condition in Theorem~\ref{thm:p-matrix-general} is based on
	the matrix
$
		T=|B^{-1}C|+|B^{-1}G|.
$
	Since
$
		T\leq |B^{-1}|\left(|C|+|G|\right),
$
	a simpler but generally more conservative sufficient condition is
$
		\rho(T_0)<1,
$
	where
$
		T_0=|B^{-1}|\left(|C|+|G|\right).
$
\end{remark}

\begin{remark}
	The term $|G|=|\gamma\Omega_4-A|$ can be estimated in two useful ways. Define
$
		\widehat{T}
		=
		|B^{-1}|
		\left(
		|C|+|\gamma\Omega_4-M|+|N|
		\right)
$
	and
$
		\widetilde{T}
		=
		|B^{-1}|
		\left(
		|C|
		+
		|\gamma\Omega_4-(M+\Omega_1-\Omega_2)|
		+
		|N+\Omega_1-\Omega_2|
		\right).
$
	Indeed, these two matrices follow from
$
		|\gamma\Omega_4-A|
		\leq
		|\gamma\Omega_4-M|+|N|
$
	and
$
		|\gamma\Omega_4-A|
		\leq
		|\gamma\Omega_4-(M+\Omega_1-\Omega_2)|
		+
		|N+\Omega_1-\Omega_2|,
$
	respectively. Therefore, either of the conditions
$
		\rho(\widehat{T})<1
$
	or
$
		\rho(\widetilde{T})<1
$
	is sufficient for the convergence of the proposed method.
\end{remark}

\begin{remark}
	Since $\rho(T)\leq \|T\|$ for any induced matrix norm, the spectral-radius
	condition in Theorem~\ref{thm:p-matrix-general} can be replaced by the
	more easily checked condition
	$
	\left\|
	|B^{-1}C|+|B^{-1}G|
	\right\|<1.
	$
	In particular, one may use the Euclidean norm $\|\cdot\|_2$, the infinity norm
	$\|\cdot\|_{\infty}$, or any other induced matrix norm.
\end{remark}

We now derive a parameter-dependent sufficient condition for the convergence of
the proposed method when the underlying splitting is of AOR type.

\begin{theorem}
	\label{thm:aor-general}
	Let $A\in\mathbb{R}^{n\times n}$ be a $P$-matrix such that
	$A=D-L-U$, where $D=dI$ with $d>0$, and $L$, $U$ are strictly
	lower and strictly upper triangular matrices, respectively. For fixed
	$\alpha>0$ and $\beta\in\mathbb{R}$, consider the AOR splitting
	$
	M=\frac{1}{\alpha}(D-\beta L)
	$
	and
	$
	N=\frac{1}{\alpha}\left\{(1-\alpha)D+(\alpha-\beta)L+\alpha U\right\}.
	$
	Let $\gamma>0$ and let $\Omega_4$ be chosen such that
	$\gamma\Omega_4=D$. Let
	$
	\Omega_1=\phi_1I,\quad \Omega_2=\phi_2L,
	$
	where $\phi_1,\phi_2\in\mathbb{R}$. Define
	$
	s_{\alpha,\beta}
	=
	\left|\left(1+\frac{1}{\alpha}\right)d+\phi_1\right|
	-
	\left|\frac{1-\alpha}{\alpha}d+\phi_1\right|
	$
	and
	$
	r_{\alpha,\beta}
	=
	\left|\frac{\beta}{\alpha}+\phi_2\right|
	+
	1+
	\left|\frac{\alpha-\beta}{\alpha}-\phi_2\right|.
	$
	If
	$
	s_{\alpha,\beta}>\rho(r_{\alpha,\beta}|L|+2|U|),
	$
	then the proposed accelerated overrelaxation (AOR)-type method converges to the unique solution of
	$\operatorname{LCP}(q,A)$ for every initial vector
	$x^{(0)}\in\mathbb{R}^n$.
\end{theorem}

\begin{proof}
	For the above AOR splitting, we have
	$
	A=M-N=D-L-U.
	$
	Using $\Omega_1=\phi_1I$, $\Omega_2=\phi_2L$, and $\gamma\Omega_4=D$, the
	matrices $B$, $C$, and $G$ become
	\begin{equation*}
		B
		=
		\left[\left(1+\frac{1}{\alpha}\right)d+\phi_1\right]I
		-
		\left(\frac{\beta}{\alpha}+\phi_2\right)L,
	\end{equation*}
	\begin{equation*}
		C
		=
		\left(\frac{1-\alpha}{\alpha}d+\phi_1\right)I
		+
		\left(\frac{\alpha-\beta}{\alpha}-\phi_2\right)L
		+
		U,
	\end{equation*}
	and
	$
	G=D-A=L+U.
	$
	Set
	$
	\xi=\left|\left(1+\frac{1}{\alpha}\right)d+\phi_1\right|,
	$
	$
	\eta=\left|\frac{\beta}{\alpha}+\phi_2\right|,
	$
	$
	p=\left|\frac{1-\alpha}{\alpha}d+\phi_1\right|
	$
	and
	$
	\tau=1+\left|\frac{\alpha-\beta}{\alpha}-\phi_2\right|.
	$
	Then
	$
	s_{\alpha,\beta}=\xi-p
	$
	and
	$
	r_{\alpha,\beta}=\eta+\tau.
	$
	
	The assumed condition implies $s_{\alpha,\beta}>0$, and hence
	$\xi>p\geq0$. Therefore $B$ has nonzero diagonal entries. Since $L$ is
	strictly lower triangular, $B$ is nonsingular.
	
	Moreover, using the disjoint sparsity patterns of $I$, $L$, and $U$, we have
	\begin{equation*}
		|C|+|G|
		\leq
		pI+\tau |L|+2|U|.
	\end{equation*}
	Let
	$
	\widetilde{B}=\xi I-\eta |L|
	$
	and
	$
	\widetilde{C}=pI+\tau |L|+2|U|.
	$
	Then $\widetilde{B}=\langle B\rangle$. Since $|L|$ is strictly lower
	triangular and $\xi>0$, $\widetilde{B}$ is a nonsingular $M$-matrix and
	$\widetilde{B}^{-1}\geq0$. By Lemma~\ref{lem:inverse-comparison},
	$
	|B^{-1}|\leq \widetilde{B}^{-1}.
	$
	Thus
	$
	0\leq |B^{-1}|(|C|+|G|)
	\leq
	\widetilde{B}^{-1}\widetilde{C}.
	$
	
	Now
	\begin{equation*}
		\widetilde{B}-\widetilde{C}
		=
		(\xi-p)I-(\eta+\tau)|L|-2|U|
		=
		s_{\alpha,\beta}I-r_{\alpha,\beta}|L|-2|U|.
	\end{equation*}
	Since $r_{\alpha,\beta}|L|+2|U|\geq0$ and
	$
	s_{\alpha,\beta}>\rho(r_{\alpha,\beta}|L|+2|U|),
	$
	Lemma~\ref{lem:m-matrix-characterization} implies that
	$\widetilde{B}-\widetilde{C}$ is a nonsingular $M$-matrix. Hence, by
	Lemma~\ref{lem:regular-splitting},
	$
	\rho(\widetilde{B}^{-1}\widetilde{C})<1.
	$
	Using Lemma~\ref{lem:spectral-monotonicity}, we obtain
	$
	\rho(|B^{-1}|(|C|+|G|))<1.
	$
	Since
	$
	|B^{-1}C|+|B^{-1}G|
	\leq
	|B^{-1}|(|C|+|G|),
	$
	another application of Lemma~\ref{lem:spectral-monotonicity} gives
	$
	\rho(|B^{-1}C|+|B^{-1}G|)<1.
	$
	The conclusion follows from Theorem~\ref{thm:p-matrix-general}.
\end{proof}

\begin{remark}
	Theorem~\ref{thm:aor-general} includes the following special cases:
	\begin{enumerate}
		\item $(\alpha,\beta)=(1,0)$: Jacobi-type proposed method, $M=D$ and
		$N=L+U$;
		\item $(\alpha,\beta)=(1,1)$: Gauss-Seidel-type proposed method,
		$M=D-L$ and $N=U$;
		\item $ \beta=\alpha \neq 0$: SOR-type proposed method,
		$M=\alpha^{-1}D-L$ and $N=(\alpha^{-1}-1)D+U$.
	\end{enumerate}
	The corresponding sufficient convergence conditions are obtained by
	substituting these values of $\alpha$ and $\beta$ into
	$s_{\alpha,\beta}$ and $r_{\alpha,\beta}$.
\end{remark}

\begin{remark}
	The condition
	$
	s_{\alpha,\beta}>\rho(r_{\alpha,\beta}|L|+2|U|)
	$
	also guarantees the nonsingularity of $B$. Indeed, it implies
	$s_{\alpha,\beta}>0$ and hence
	$
	\left(1+\frac{1}{\alpha}\right)d+\phi_1\neq0.
	$
	Since $B$ is lower triangular with this quantity on its diagonal, $B$ is
	nonsingular. In the Gauss-Seidel case $(\alpha,\beta)=(1,1)$, this
	nonsingularity requirement reduces to
	$
	2d+\phi_1\neq0.
	$
\end{remark}

Since the spectral radius is bounded above by any induced matrix norm, we obtain
the following simple sufficient condition.

\begin{corollary}
	\label{cor:norm-aor}
	Under the assumptions of Theorem~\ref{thm:aor-general}, let
	$\|\cdot\|$ be any induced matrix norm. If
	$
	s_{\alpha,\beta}>
	\left\|r_{\alpha,\beta}|L|+2|U|\right\|,
	$
	then the proposed AOR-type method converges to the unique solution of
	$\operatorname{LCP}(q,A)$ for every initial vector
	$x^{(0)}\in\mathbb{R}^n$.
\end{corollary}

The following result provides a useful spectral-radius identity for certain
structured lower and upper triangular matrix pairs.

\begin{lemma}
	\label{lem:spectral-scaling}
	Let $L$ and $U$ be strictly lower and strictly upper triangular matrices,
	respectively. For $a,b>0$, the identity
	\begin{equation}
		\rho(a|L|+b|U|)
		=
		\sqrt{ab}\,\rho(|L|+|U|)
		\label{eq:spectral-scaling}
	\end{equation}
	holds if either of the following conditions is satisfied:
	\begin{enumerate}
		\item There exists a mapping
		$\ell:\{1,2,\ldots,n\}\to\mathbb{Z}$ such that, for all $i,j$,
		$
		L_{ij}\neq0 \Longrightarrow \ell(i)-\ell(j)=1
		$
		and
		$
		U_{ij}\neq0 \Longrightarrow \ell(i)-\ell(j)=-1.
		$
		
		\item The matrices $|L|$ and $|U|$ have a common block upper triangular
		structure with respect to the same block partition, and every pair of
		corresponding diagonal blocks $|L_i|$ and $|U_i|$ satisfies
		$
		\rho(a|L_i|+b|U_i|)
		=
		\sqrt{ab}\,\rho(|L_i|+|U_i|).
		$
	\end{enumerate}
\end{lemma}

\begin{proof}
	For case $(1)$, let $t=\sqrt{a/b}$ and define
	$
	P=\operatorname{diag}(t^{\ell(1)},t^{\ell(2)},\ldots,t^{\ell(n)}).
	$
	The matrix $P$ is nonsingular because $t>0$. If $L_{ij}\neq0$, then
	$|L|_{ij}\neq0$ and $\ell(i)-\ell(j)=1$. Hence
	$
	(P^{-1}|L|P)_{ij}
	=
	t^{-\ell(i)}|L|_{ij}t^{\ell(j)}
	=
	t^{-1}|L|_{ij}.
	$
	Thus $P^{-1}|L|P=t^{-1}|L|$. Similarly, if $U_{ij}\neq0$, then
	$|U|_{ij}\neq0$ and $\ell(i)-\ell(j)=-1$, and therefore
	$
	(P^{-1}|U|P)_{ij}
	=
	t^{-\ell(i)}|U|_{ij}t^{\ell(j)}
	=
	t|U|_{ij}.
	$
	Hence $P^{-1}|U|P=t|U|$. Consequently,
	\begin{equation*}
		P^{-1}(a|L|+b|U|)P
		=
		at^{-1}|L|+bt|U|
		=
		\sqrt{ab}(|L|+|U|).
	\end{equation*}
	Therefore, $a|L|+b|U|$ is similar to
	$\sqrt{ab}(|L|+|U|)$, which proves \eqref{eq:spectral-scaling}.
	
	For case $(2)$, both $a|L|+b|U|$ and $|L|+|U|$ are block upper triangular with
	diagonal blocks $a|L_i|+b|U_i|$ and $|L_i|+|U_i|$, respectively. Since the
	eigenvalues of a block upper triangular matrix are determined by its diagonal
	blocks, we have
	\begin{equation*}
		\rho(a|L|+b|U|)
		=
		\max_i\rho(a|L_i|+b|U_i|)
		=
		\sqrt{ab}\max_i\rho(|L_i|+|U_i|)
		=
		\sqrt{ab}\,\rho(|L|+|U|).
	\end{equation*}
\end{proof}

\begin{remark}
	Examples~\ref{Ex1}, \ref{Ex2}, and \ref{Ex4} satisfy condition $(1)$ of
	Lemma~\ref{lem:spectral-scaling}. For an $m\times m$ grid, if the index
	$k$ corresponds to the point $(i,j)$ through $k=(i-1)m+j$, then one may
	choose
	$
	\ell(k)=i+j-2.
	$
	Example~\ref{Ex3} satisfies condition $(2)$. With respect to the natural
	block partition, the matrices $|L|$ and $|U|$ have a common block upper
	triangular structure. The corresponding diagonal block pairs satisfy
	condition $(1)$, while the additional first and second block super-diagonal
	terms occur only in off-diagonal blocks and therefore do not affect the
	eigenvalues. Hence, for all four examples,
	$
	\rho(a|L|+b|U|)=\kappa_m\sqrt{ab},
~
	\kappa_m=\rho(|L|+|U|).
	$
\end{remark}

\begin{remark}
	For the test problems in Section~\ref{sec5}, the coefficient matrices are
	written as
	$
	A=D-L-U,
	$
	where
	$
	L=-\operatorname{tril}(A,-1),~
	U=-\operatorname{triu}(A,1).
	$
	Since the off-diagonal entries of these matrices are non-positive, we have
	$L\geq0$ and $U\geq0$. Hence, for these examples,
	$
	|L|=L,~ |U|=U,
	$
	and therefore
	$
	\kappa_m=\rho(|L|+|U|)=\rho(L+U).
	$
\end{remark}

Using Lemma~\ref{lem:spectral-scaling}, the condition in
Theorem~\ref{thm:aor-general} can be expressed as explicit intervals for the
parameters $\phi_1$ and $\phi_2$.

\begin{theorem}
	\label{thm:parameter-interval}
	Suppose that the assumptions of Theorem~\ref{thm:aor-general} hold and that
	the pair \((|L|,|U|)\) satisfies the spectral-scaling property in
	Lemma~\ref{lem:spectral-scaling}. Let
	$\kappa_m:=\rho(|L|+|U|)>0$ and assume that $d>\kappa_m$. For fixed
	$\alpha>0$ and $\beta\in\mathbb{R}$, the proposed AOR-type method converges
	for every initial vector $x^{(0)}\in\mathbb{R}^n$ if either of the following
	conditions holds:
	\begin{enumerate}
		\item
		\begin{equation}
			\kappa_m-\frac{d}{\alpha}
			<\phi_1<
			d\left(1-\frac{1}{\alpha}\right),
			\label{eq:phi1-case-a}
		\end{equation}
		and
		\begin{equation}
			1-\frac{(d/\alpha+\phi_1)^2}{\kappa_m^2}
			-\frac{\beta}{\alpha}
			<\phi_2<
			\frac{(d/\alpha+\phi_1)^2}{\kappa_m^2}
			-\frac{\beta}{\alpha}.
			\label{eq:phi2-case-a}
		\end{equation}
		
		\item
		\begin{equation}
			\phi_1\geq d\left(1-\frac{1}{\alpha}\right),
			\label{eq:phi1-case-b}
		\end{equation}
		and
		\begin{equation}
			1-\frac{d^2}{\kappa_m^2}-\frac{\beta}{\alpha}
			<\phi_2<
			\frac{d^2}{\kappa_m^2}-\frac{\beta}{\alpha}.
			\label{eq:phi2-case-b}
		\end{equation}
	\end{enumerate}
\end{theorem}

\begin{proof}
	Define
	$
	\theta=\frac{1-\alpha}{\alpha}d+\phi_1
	$
	and
	$
	y=\frac{\beta}{\alpha}+\phi_2.
	$
	Then
	$
	s_{\alpha,\beta}=|\theta+2d|-|\theta|
	$
	and
	$
	r_{\alpha,\beta}=|y|+1+|1-y|.
	$
	
	By Lemma~\ref{lem:spectral-scaling},
	$
	\rho(r_{\alpha,\beta}|L|+2|U|)
	=
	\kappa_m\sqrt{2r_{\alpha,\beta}}.
	$
	Hence, the condition of Theorem~\ref{thm:aor-general} is satisfied whenever
	$
	s_{\alpha,\beta}>
	\kappa_m\sqrt{2r_{\alpha,\beta}}.
	$
	
	Notice that $r_{\alpha,\beta}\geq2$ and $\kappa_m>0$. Therefore, the
	right-hand side of the preceding inequality is positive, and the
	inequality implies $s_{\alpha,\beta}>0$. It may consequently be squared
	without changing its direction. Thus, it is enough to require
	\begin{equation}
		r_{\alpha,\beta}
		<
		\frac{s_{\alpha,\beta}^{2}}{2\kappa_m^{2}}.
		\label{eq:r-bound-cor}
	\end{equation}
	
	The minimum value of $r_{\alpha,\beta}=|y|+1+|1-y|$ is $2$. Since the
	inequality in \eqref{eq:r-bound-cor} is strict, admissible values of $y$
	can exist only when $s_{\alpha,\beta}>2\kappa_m$. Set
	$
	c=s_{\alpha,\beta}^{2}/(2\kappa_m^{2}).
	$
	Then $c>2$.
	
	We now solve $|y|+1+|1-y|<c$. If $y<0$, the left-hand side is $2-2y$,
	and hence $y>1-c/2$. For $0\leq y\leq1$, it is equal to $2$, so the
	inequality holds because $c>2$. If $y>1$, it is equal to $2y$, giving
	$y<c/2$. Combining the three cases gives
	$
	1-c/2<y<c/2.
	$
	Since $y=\beta/\alpha+\phi_2$, we obtain
	\begin{equation}
		1-\frac{s_{\alpha,\beta}^{2}}{4\kappa_m^{2}}
		-\frac{\beta}{\alpha}
		<\phi_2<
		\frac{s_{\alpha,\beta}^{2}}{4\kappa_m^{2}}
		-\frac{\beta}{\alpha}.
		\label{eq:general-phi2-interval}
	\end{equation}
	
	It remains to express the condition $s_{\alpha,\beta}>2\kappa_m$ in
	terms of $\phi_1$. The value of $s_{\alpha,\beta}$ depends on the
	position of $\theta$ relative to $-2d$ and $0$. If
	$\theta\leq-2d$, then $s_{\alpha,\beta}=-2d<0$, so this case cannot
	satisfy the convergence condition. Hence, only the following two cases
	need to be considered.
	
	First, let $-2d<\theta<0$. Direct evaluation of the absolute values gives
	$
	s_{\alpha,\beta}=2(d+\theta)=2(d/\alpha+\phi_1).
	$
	Therefore, $s_{\alpha,\beta}>2\kappa_m$ is equivalent to
	$
	\phi_1>\kappa_m-d/\alpha.
	$
	Moreover, $\theta<0$ is equivalent to
	$
	\phi_1<d(1-1/\alpha).
	$
	Thus,
	\begin{equation*}
		\kappa_m-\frac{d}{\alpha}
		<\phi_1<
		d\left(1-\frac{1}{\alpha}\right).
	\end{equation*}
	The omitted inequality $\theta>-2d$ follows automatically from the lower
	bound above, since it gives $\theta>\kappa_m-d>-2d$. Substituting
	$s_{\alpha,\beta}=2(d/\alpha+\phi_1)$ into
	\eqref{eq:general-phi2-interval} gives \eqref{eq:phi2-case-a}.
	
	Second, let $\theta\geq0$. In this case,
	$
	s_{\alpha,\beta}=2d.
	$
	The condition $\theta\geq0$ is equivalent to
	$
	\phi_1\geq d(1-1/\alpha),
	$
	while $s_{\alpha,\beta}>2\kappa_m$ is exactly $d>\kappa_m$, which is
	already assumed. Substituting $s_{\alpha,\beta}=2d$ into
	\eqref{eq:general-phi2-interval} gives \eqref{eq:phi2-case-b}.
	
	Therefore, each of the two stated sets of parameter conditions implies the
	sufficient condition of Theorem~\ref{thm:aor-general}. Hence, the proposed
	AOR-type method converges for every initial vector.
\end{proof}

\begin{remark}
	For $(\alpha,\beta)=(1,1)$, the AOR splitting reduces to the
	Gauss-Seidel splitting $M=D-L$ and $N=U$. Therefore, assuming
	$d>\kappa_m$, the proposed Gauss-Seidel-type method RSMGS converges if
	$(\phi_1,\phi_2)$ satisfies either of the following cases:
	\begin{equation*}
		\begin{cases}
			\kappa_m-d<\phi_1<0,\\[1mm]
			-\dfrac{(d+\phi_1)^2}{\kappa_m^2}
			<\phi_2<
			\dfrac{(d+\phi_1)^2}{\kappa_m^2}-1,
		\end{cases}
		\qquad\text{or}\qquad
		\begin{cases}
			\phi_1\geq0,\\[1mm]
			-\dfrac{d^2}{\kappa_m^2}
			<\phi_2<
			\dfrac{d^2}{\kappa_m^2}-1.
		\end{cases}
	\end{equation*}
	where $\kappa_m=\rho(|L|+|U|)$.
\end{remark}

We next give a convergence result for the proposed method when the coefficient
matrix belongs to the class of $H_{+}$-matrices. The condition is stated in a weighted componentwise form, which gives a directly verifiable sufficient criterion.
criterion.

\begin{theorem}
	\label{thm:Hplus-weighted}
	Let $A\in\mathbb{R}^{n\times n}$ be an $H_{+}$-matrix and let $A=M-N$ be a
	splitting of $A$. Let $\gamma>0$, let $\Omega_4$ be a positive diagonal
	matrix, and define
	$
	B=M+\Omega_1-\Omega_2+\gamma\Omega_4,
	$
	$
	C=N+\Omega_1-\Omega_2
	$
	and
	$
	G=\gamma\Omega_4-A.
	$
	Assume that $B$ is an $H$-matrix. If there exists a vector $w>0$ such that
	\begin{equation}
		(|C|+|G|)w<\langle B\rangle w,
		\label{eq:Hplus-weighted-condition}
	\end{equation}
	then the proposed method converges to the unique solution of
	$\operatorname{LCP}(q,A)$ for every initial vector
	$x^{(0)}\in\mathbb{R}^n$.
\end{theorem}

\begin{proof}
	Set
	$
	Q=|C|+|G|.
	$
	Then $Q\geq0$. By \eqref{eq:Hplus-weighted-condition},
	$
	(\langle B\rangle-Q)w>0
	$
	for some $w>0$. Since $\langle B\rangle-Q$ is a $Z$-matrix,
	Lemma~\ref{lem:M-semipositive} implies that $\langle B\rangle-Q$ is a
	nonsingular $M$-matrix.
	
	Moreover, because $B$ is an $H$-matrix, $\langle B\rangle$ is a
	nonsingular $M$-matrix and hence $\langle B\rangle^{-1}\geq0$. Therefore,
	$\langle B\rangle-Q$ is a regular splitting with splitting matrix
	$\langle B\rangle$ and iteration matrix $\langle B\rangle^{-1}Q$. By
	Lemma~\ref{lem:regular-splitting}, we obtain
	\begin{equation*}
		\rho(\langle B\rangle^{-1}Q)<1.
	\end{equation*}
	
	Since $B$ is an $H$-matrix, Lemma~\ref{lem:inverse-comparison} gives
	$
	|B^{-1}|\leq \langle B\rangle^{-1}.
	$
	Therefore,
	$
	0\leq |B^{-1}|Q\leq \langle B\rangle^{-1}Q.
	$
	Using Lemma~\ref{lem:spectral-monotonicity}, we get
	$
	\rho(|B^{-1}|Q)<1.
	$
	
	Finally, since
	$
	|B^{-1}C|+|B^{-1}G|
	\leq
	|B^{-1}|(|C|+|G|)
	=
	|B^{-1}|Q,
	$
	another application of Lemma~\ref{lem:spectral-monotonicity} yields
	$
	\rho(|B^{-1}C|+|B^{-1}G|)<1.
	$
	Since every $H_{+}$-matrix is a $P$-matrix, the conclusion follows from
	Theorem~\ref{thm:p-matrix-general}.
\end{proof}

The next corollary gives easily verifiable sufficient conditions under an
$H$-compatible splitting.

\begin{corollary}
	\label{cor:Hplus-structural}
	Let $A=D-E\in\mathbb{R}^{n\times n}$ be an $H_{+}$-matrix, where
	$D=\operatorname{diag}(A)>0$ and $E=D-A$. Let
	$
	A=(M+\Omega_1-\Omega_2)-(N+\Omega_1-\Omega_2)
	$
	be an $H$-compatible splitting, that is,
	\begin{equation}
		\langle A\rangle
		=
		\langle M+\Omega_1-\Omega_2\rangle
		-
		|N+\Omega_1-\Omega_2|.
		\label{eq:H-compatible}
	\end{equation}
	Let $\gamma>0$ and let $\Omega_4$ be a positive diagonal matrix. Assume
	that the diagonal entries of $M+\Omega_1-\Omega_2$ are positive. Then the
	proposed method converges to the unique solution of $\operatorname{LCP}(q,A)$
	for every initial vector $x^{(0)}\in\mathbb{R}^n$ if either of the following
	conditions holds:
	\begin{enumerate}
		\item $\gamma\Omega_4\geq D$;
		\item $0<\gamma\Omega_4<D$ and
		$\langle A\rangle+2\gamma\Omega_4-D-|E|$ is a nonsingular $M$-matrix.
	\end{enumerate}
\end{corollary}

\begin{proof}
	Set
	$
	B=M+\Omega_1-\Omega_2+\gamma\Omega_4,
	$
	$
	C=N+\Omega_1-\Omega_2
	$
	and
	$
	Q=|C|+|\gamma\Omega_4-A|.
	$
	Since the diagonal entries of $M+\Omega_1-\Omega_2$ are positive and
	$\gamma\Omega_4$ is diagonal with positive diagonal entries, we have
	$
	\langle B\rangle
	=
	\langle M+\Omega_1-\Omega_2\rangle+\gamma\Omega_4.
	$
	
	First suppose that $\gamma\Omega_4\geq D$. Since $A=D-E$ and the diagonal
	and off-diagonal parts have disjoint supports, we have
	$
	|\gamma\Omega_4-A|=\gamma\Omega_4-D+|E|.
	$
	Using \eqref{eq:H-compatible}, we obtain
	$
	\langle B\rangle-Q=2\langle A\rangle.
	$
	Since $A$ is an $H_{+}$-matrix, $\langle A\rangle$ is a nonsingular
	$M$-matrix. Hence $\langle B\rangle-Q$ is a nonsingular $M$-matrix.
	
	Next suppose that $0<\gamma\Omega_4<D$. Then
	$
	|\gamma\Omega_4-A|=D-\gamma\Omega_4+|E|,
	$
	and therefore
	$
	\langle B\rangle-Q
	=
	\langle A\rangle+2\gamma\Omega_4-D-|E|.
	$
	By assumption, this is a nonsingular $M$-matrix.
	
	Thus, in both cases, $\langle B\rangle-Q$ is a nonsingular $M$-matrix.
	By Lemma~\ref{lem:M-semipositive}, there exists a vector $w>0$ such that
	$
	(\langle B\rangle-Q)w>0.
	$
	Equivalently,
	$
	Qw<\langle B\rangle w.
	$
	
	Also, since $Q\geq0$, we have
	$
	\langle B\rangle w=(\langle B\rangle-Q)w+Qw>0.
	$
	As $\langle B\rangle$ is a $Z$-matrix, Lemma~\ref{lem:M-semipositive}
	implies that $\langle B\rangle$ is a nonsingular $M$-matrix. Therefore,
	$B$ is an $H$-matrix.
	
	Thus, all assumptions of Theorem~\ref{thm:Hplus-weighted} are satisfied,
	and the result follows.
\end{proof}

\begin{remark}
	\label{Rem4.8}
	The preceding results provide admissible intervals for the parameters
	$\phi_1$ and $\phi_2$. For practical computations, it is also useful to
	measure the strength of the sufficient convergence condition. Following the
	comparison matrix used in the proof of Theorem~\ref{thm:aor-general}, define
	$
	\mathcal{R}_{\alpha,\beta}(\phi_1,\phi_2)
	=
	\rho(\widetilde{B}^{-1}\widetilde{C}),
	$
	where
	$
	\widetilde{B}=\xi I-\eta |L|
	$
	and
	$
	\widetilde{C}=pI+\tau |L|+2|U|.
	$
	Here
	\begin{equation*}
		\xi=\left|\left(1+\frac{1}{\alpha}\right)d+\phi_1\right|,\quad
		\eta=\left|\frac{\beta}{\alpha}+\phi_2\right|,\quad
		p=\left|\frac{1-\alpha}{\alpha}d+\phi_1\right|,\quad
		\tau=1+\left|\frac{\alpha-\beta}{\alpha}-\phi_2\right|.
	\end{equation*}
	The quantity $\mathcal{R}_{\alpha,\beta}(\phi_1,\phi_2)$ is called the
	comparison convergence factor. If
	$
	\mathcal{R}_{\alpha,\beta}(\phi_1,\phi_2)<1,
	$
	then the sufficient condition in Theorem~\ref{thm:aor-general} is satisfied.
	
	When the pair $(|L|,|U|)$ satisfies Lemma~\ref{lem:spectral-scaling}, this
	condition can also be checked through the scalar convergence margin
	\begin{equation*}
		\Delta_{\alpha,\beta}(\phi_1,\phi_2)
		=
		s_{\alpha,\beta}
		-
		\kappa_m\sqrt{2r_{\alpha,\beta}},
		\qquad
		\kappa_m=\rho(|L|+|U|).
	\end{equation*}
	Thus, any parameter pair satisfying
	$
	\Delta_{\alpha,\beta}(\phi_1,\phi_2)>0
	$
	belongs to the admissible region described in
	Theorem~\ref{thm:parameter-interval}.
	
	For the Gauss-Seidel-type method $(\alpha,\beta)=(1,1)$, the margin becomes
	\begin{equation*}
		\Delta_{1,1}(\phi_1,\phi_2)
		=
		|2d+\phi_1|-|\phi_1|
		-
		\kappa_m
		\sqrt{
			2\left(|1+\phi_2|+1+|\phi_2|\right)
		}.
	\end{equation*}
	Therefore, the theoretically admissible parameters for the proposed
	RSMGS method are those satisfying
	$
	\Delta_{1,1}(\phi_1,\phi_2)>0.
	$
	
	The margin $\Delta_{\alpha,\beta}$ should be interpreted as a sufficient
	convergence margin, not as the exact convergence rate of the proposed method.
	Since its derivation is based on componentwise absolute-value estimates, it
	may not capture possible cancellations in the actual iteration matrix.
	Therefore, parameters that give better practical performance may differ from
	those suggested by maximizing this margin.
	
	In the numerical section, Algorithm~\ref{algo2} is used to search for
	experimentally effective parameters. The admissible intervals in
	Theorem~\ref{thm:parameter-interval} and the margin
	$\Delta_{\alpha,\beta}(\phi_1,\phi_2)$ provide a theoretical guide for this
	search, while the final parameter pairs are selected according to the observed
	iteration counts and CPU times.
\end{remark}

\section{Experimental Analysis}\label{sec5}

In this section, we examine the numerical performance of the proposed
relaxed shifted matrix-splitting modulus-based Gauss-Seidel method for solving
large-scale sparse $\operatorname{LCP}(q,A)$. The efficiency of the method is
measured by the number of iterations (IT), CPU time in seconds (CPU), and the
final residual (RES). The residual is computed as
$
	\operatorname{RES}(z^{(k)})
	=
	\left\|
	\min\left(Az^{(k)}+q,z^{(k)}\right)
	\right\|_2,
$
where the minimum is taken componentwise. The stopping criterion is
$\operatorname{RES}(z^{(k)})<\varepsilon=10^{-5}$.

For all numerical tests, the initial vector is chosen as
$
	x^{(0)}=(1,0,1,0,\ldots,1,0,\ldots)^T\in\mathbb{R}^n.
$
In the proposed method, we use the Gauss-Seidel splitting $M=D-L$ and $N=U$.
We further take
$
	\Omega_1=\phi_1I,~
	\Omega_2=\phi_2L,~
	\gamma=2,~
	\Omega_4=\frac{1}{2}D.
$
Thus, $\gamma\Omega_4=D$, and Method~\ref{meth:rsmms} becomes
\begin{equation}
	\left(2D+\phi_1I-(1+\phi_2)L\right)x^{(k+1)}
	=
	\left(U+\phi_1I-\phi_2L\right)x^{(k)}
	+
	(D-A)|x^{(k)}|
	-
	2q.
	\label{eq:rsmgs-numerical}
\end{equation}
This scheme is denoted by \textbf{RSMGS}. The choice $\gamma=2$ and
$\Omega_4=\frac{1}{2}D$ is consistent with commonly used modulus-based
Gauss-Seidel-type methods and has been observed to give good practical
convergence behavior; see \cite{Xu2015,RenWangTangWang2019,Cvetkovic2014}.

For the SOR-type comparison methods, the optimal values of parameters are selected
experimentally. For the RGTMSOR method, we fix $\tau=0.9$, as used in \cite{LiWangLiu2022}, while
the parameters $\theta$, and $\omega$ are selected experimentally. For RSMGS, the
pairs $(\phi_1,\phi_2)$ are selected using the admissible intervals derived in
Section~\ref{sec4} together with the experimental parameter search presented
in Subsection~\ref{ss5.2}.

\begin{table}[h]
	\centering
	\caption{Testing methods used for comparison}
	\label{tab:testing-methods}
	\small
	\renewcommand{\arraystretch}{1.2}
	\begin{tabular}{@{}lll@{}}
		\toprule
		Method & Ref. & Description \\
		\midrule
		MGS & \cite{Bai2010} & Modulus-based matrix-splitting Gauss-Seidel method \\
		MSOR & \cite{Bai2010} & Modulus-based matrix-splitting successive overrelaxation method \\
		MB-DS$_L$ & \cite{Fang2019} & Modulus-based matrix double-splitting method \\
		NMGS & \cite{Wu2022} & New modulus-based Gauss-Seidel method \\
		NPGS & \cite{Das2025} & New projected Gauss-Seidel iteration method \\
		NPSOR & \cite{Das2025} & New projected successive overrelaxation iteration method \\
		RGTMSOR & \cite{LiWangLiu2022} & Relaxed generalized two-sweep modulus-based SOR method \\
		\textbf{RSMGS} &  & Relaxed shifted matrix-splitting modulus-based Gauss-Seidel method \\
		\botrule
	\end{tabular}
\end{table}

All numerical experiments were carried out in MATLAB R2025b using double
precision arithmetic on a Windows 11, 64-bit operating system. The computations were performed on a system equipped with a 13th Gen Intel(R) Core(TM) i7 processor with 16 cores and 24 logical processors, and 16 GB RAM.

\subsection{Test Problems}

We consider four benchmark problems to test the performance of the proposed
method. In all examples, $m$ is a prescribed positive integer and $n=m^2$.

\begin{example}[Symmetric block tridiagonal problem \cite{Bai2010}]\label{Ex1}
	We consider the standard block tridiagonal LCP test problem. The coefficient matrix is defined by $A=\widehat{A}+\mu I_n$, where
	$\widehat{A}=\operatorname{Tridiag}(-I_m,S,-I_m)\in\mathbb{R}^{n\times n}$
	and $S=\operatorname{tridiag}(-1,4,-1)\in\mathbb{R}^{m\times m}$.
	Here, $I_n$ and $I_m$ denote the identity matrices of order $n$ and $m$,
	respectively. The vector $q$ is generated by $q=-Az^\ast$, where
	$z^\ast=(1,2,1,2,\ldots, 1,2,\ldots)^T\in\mathbb{R}^n$ is the unique solution of the
	corresponding $\operatorname{LCP}(q,A)$.
\end{example}

\begin{example}[Non-symmetric block tridiagonal problem \cite{Bai2010}]\label{Ex2}
	We consider a non-symmetric variant of the preceding test problem. The coefficient matrix is defined by
	$A=\widehat{A}+\mu I_n$, where
	$\widehat{A}=\operatorname{Tridiag}(-1.5I_m,S,-0.5I_m)
	\in\mathbb{R}^{n\times n}$ and
	$S=\operatorname{tridiag}(-1.5,4,-0.5)\in\mathbb{R}^{m\times m}$.
	Here, $I_n$ and $I_m$ denote the identity matrices of order $n$ and $m$,
	respectively. The vector $q$ is chosen as $q=-Az^\ast$, where
	$z^\ast=(1,2,1,2,\ldots, 1,2,\ldots)^T\in\mathbb{R}^n$ is the unique solution of the
	corresponding $\operatorname{LCP}(q,A)$.
\end{example}

\begin{example}[Sparse upper block problem \cite{Das2025}]\label{Ex3}
	Consider a sparse non-symmetric matrix of the form $A=P+\mu I_n$, 
	where $P\in\mathbb{R}^{n\times n}$ is defined by
	$$
	P=I_m\otimes S+Y+J
	=
	\begin{pmatrix}
		S & yI_m & jI_m &        &        & 0\\
		0 & S  & yI_m & jI_m     &        &  \\
		& 0  & S  & yI_m     & jI_m     &  \\
		&    & \ddots & \ddots & \ddots & \ddots\\
		&    &        & 0 & S & yI_m\\
		0 &    & \cdots &   & 0 & S
	\end{pmatrix}
	\in\mathbb{R}^{n\times n}.
	$$
	Here, $I_n$ and $I_m$ denote the identity matrices of order $n$ and $m$,
	respectively. The matrices $Y$ and $J$ denote the first and second block
	superdiagonal matrices, respectively. Moreover,
	$$
	S=\operatorname{tridiag}(-1,4,-1)
	=
	\begin{pmatrix}
		4 & -1 &        &        & 0\\
		-1 & 4 & -1     &        &  \\
		& -1 & 4     & \ddots &  \\
		&    & \ddots & \ddots & -1\\
		0  &    &        & -1 & 4
	\end{pmatrix}
	\in\mathbb{R}^{m\times m}.
	$$
	In the computations, we take $y=j=-1$. The vector $q$ is defined by
	$q=-Az^\ast$, where
$z^\ast=(1,2,1,2,\ldots, 1,2,\ldots)^T\in\mathbb{R}^n$ is the unique solution of the
	corresponding $\operatorname{LCP}(q,A)$.
\end{example}

\begin{example}[Quasi-variational inequality problem \cite{Isac2006,Wu2022}]\label{Ex4}
	This example is derived from a quasi-variational inequality model associated
	with a continuous optimal control problem. In its
	abstract form, the problem is to determine $u\in\mathcal{K}(u)$ such that
	$$
	(r-u)^T(Au+\mathcal{F}(u))\geq 0,\qquad \forall r\in\mathcal{K}(u),
	$$
	where $\mathcal{K}(u)=\psi(u)+K\subset\mathbb{R}^n$, $K$ is a positive cone,
	$\psi$ is an implicit obstacle mapping, and $\mathcal{F}$ is a given
	vector-valued mapping. This model can be reformulated as an
	$\operatorname{LCP}(q,A)$.
	In the resulting LCP, the coefficient matrix is the same as in
	Example~\ref{Ex1}.
	Here, the vector $\mathcal{F}(u)=q=(-1,1,-1,1,\ldots, -1, 1, \ldots)^T\in\mathbb{R}^n$ is used, and we take $r=2u$ in the computations.
\end{example}

\begin{table}[htbp]
	\centering
	\caption{Performance comparison for Example~\ref{Ex1}}
	\label{tab:ex1}
	\scriptsize
	\renewcommand{\arraystretch}{1.02}
	\setlength{\tabcolsep}{2pt}
	\begin{tabular*}{\textwidth}{@{\extracolsep{\fill}}l@{}clccccc@{}}
		\toprule
		\multirow{2}{*}{} & \multirow{2}{*}{Method} & \multirow{2}{*}{}
		& \multicolumn{5}{c}{$m$} \\
		\cmidrule(lr){4-8}
		& & & 100 & 400 & 600 & 800 & 1000 \\
		\midrule
		
		\multirow{24}{*}{$\mu=4$}
		& \multirow{3}{*}{MGS}
		& IT  & 42 & 44 & 45 & 45 & 45 \\
		& & CPU & 0.0455 & 0.4918 & 1.1187 & 2.0303 & 3.2073 \\
		& & RES & 8.3907e-06 & 8.7680e-06 & 7.6526e-06 & 8.8551e-06 & 9.9128e-06 \\
		
		\addlinespace
		& \multirow{3}{*}{\shortstack{MSOR\\[-0.5mm]$\alpha=0.85$}}
		& IT  & 19 & 20 & 21 & 21 & 21 \\
		& & CPU & 0.0150 & 0.2498 & 0.6300 & 0.9722 & 1.5429 \\
		& & RES & 4.3252e-06 & 6.9453e-06 & 4.0024e-06 & 5.3513e-06 & 6.7001e-06 \\
		
		\addlinespace
		& \multirow{3}{*}{MB-DS$_L$}
		& IT  & 16 & 17 & 17 & 18 & 18 \\
		& & CPU & 0.9074 & 14.2992 & 31.3217 & 51.8339 & 212.7356 \\
		& & RES & 4.3248e-06 & 6.0817e-06 & 9.1791e-06 & 4.0902e-06 & 5.1226e-06 \\
		
		\addlinespace
	& \multirow{3}{*}{\shortstack{NMGS\\[-0.5mm]}}
	& IT  & 21 & 23 & 23 & 24 & 24 \\
	& & CPU & 0.0242 & 0.2076 & 0.3815 & 0.7016 & 1.1178 \\
	& & RES & 6.9784e-06 & 5.4040e-06 & 8.1558e-06 & 4.6735e-06 & 5.8529e-06 \\
		
		\addlinespace
		& \multirow{3}{*}{NPGS}
		& IT  & 25 & 27 & 28 & 28 & 28 \\
		& & CPU & 0.0226 & 0.2009 & 0.4288 & 0.7313 & 1.1427 \\
		& & RES & 6.9580e-06 & 7.3060e-06 & 5.5090e-06 & 7.3665e-06 & 9.2240e-06 \\
		
		\addlinespace
		& \multirow{3}{*}{\shortstack{NPSOR\\[-0.5mm]$\alpha=1.7$}}
		& IT  & 17 & 19 & 19 & 19 & 19 \\
		& & CPU & 0.0075 & 0.1421 & 0.2802 & 0.5104 & 0.7795 \\
		& & RES & 7.6569e-06 & 3.9831e-06 & 5.8352e-06 & 7.6849e-06 & 9.5336e-06 \\
		
		\addlinespace
	& \multirow{3}{*}{\shortstack{RGTMSOR\\[-0.5mm]$\alpha=0.9$\\[-0.5mm]$\theta=1.2,\omega=0.2$}}
	& IT  & 14 & 14 & 15 & 15 & 15 \\
	& & CPU & 0.0138 & 0.2092 & 0.4533 & 0.7902 & 1.1911 \\
	& & RES & 5.1630e-06 & 9.9128e-06 & 3.8819e-06 & 4.4665e-06 & 4.9896e-06 \\
		
		\addlinespace
		& \multirow{3}{*}{\shortstack{\textbf{RSMGS}\\[-0.5mm]$(\phi_1,\phi_2)=(-1,1.75)$}}
		& IT  & 11 & 11 & 11 & 12 & 12 \\
		& & CPU & 0.0152 & 0.1158 & 0.2410 & 0.4538 & 0.7231 \\
		& & RES & 8.1704e-06 & 9.1263e-06 & 9.9000e-06 & 3.0183e-06 & 3.1300e-06 \\
		
		\midrule
		
		\multirow{24}{*}{$\mu=6$}
		& \multirow{3}{*}{MGS}
		& IT  & 33 & 35 & 35 & 35 & 36 \\
		& & CPU & 0.0229 & 0.4296 & 1.1089 & 1.6604 & 2.6437 \\
		& & RES & 8.3210e-06 & 7.0304e-06 & 8.6368e-06 & 9.9882e-06 & 7.1713e-06 \\
		
		\addlinespace
		& \multirow{3}{*}{\shortstack{MSOR\\[-0.5mm]$\alpha=0.85$}}
		& IT  & 15 & 16 & 16 & 16 & 16 \\
		& & CPU & 0.0105 & 0.2031 & 0.4664 & 0.8905 & 1.1190 \\
		& & RES & 3.6826e-06 & 3.8374e-06 & 5.6339e-06 & 7.4284e-06 & 9.2221e-06 \\
		
		\addlinespace
		& \multirow{3}{*}{MB-DS$_L$}
		& IT  & 13 & 14 & 14 & 14 & 15 \\
		& & CPU & 0.3096 & 5.1028 & 11.5758 & 44.1732 & 148.2979 \\
		& & RES & 4.2078e-06 & 4.3999e-06 & 6.6351e-06 & 8.8702e-06 & 2.7751e-06 \\
		
		\addlinespace
	& \multirow{3}{*}{\shortstack{NMGS\\[-0.5mm]}}
	& IT  & 17 & 18 & 18 & 19 & 19 \\
	& & CPU & 0.0151 & 0.1589 & 0.3913 & 0.6695 & 1.0596 \\
	& & RES & 4.3672e-06 & 6.0984e-06 & 9.1958e-06 & 4.0966e-06 & 5.1291e-06 \\
		\addlinespace
		& \multirow{3}{*}{NPGS}
		& IT  & 19 & 21 & 21 & 22 & 22 \\
		& & CPU & 0.0068 & 0.1598 & 0.3186 & 0.5738 & 0.9277 \\
		& & RES & 9.7593e-06 & 6.5033e-06 & 9.8018e-06 & 5.2390e-06 & 6.5584e-06 \\
		
		\addlinespace
		& \multirow{3}{*}{\shortstack{NPSOR\\[-0.5mm]$\alpha=1.5$}}
		& IT  & 14 & 15 & 15 & 15 & 16 \\
		& & CPU & 0.0058 & 0.1095 & 0.2337 & 0.4061 & 0.6606 \\
		& & RES & 5.0989e-06 & 4.9721e-06 & 7.1718e-06 & 9.3645e-06 & 3.3345e-06 \\
		
		\addlinespace
	& \multirow{3}{*}{\shortstack{RGTMSOR\\[-0.5mm]$\alpha=1$\\[-0.5mm]$\theta=0.8,\omega=0.1$}}
	& IT  & 12 & 13 & 13 & 14 & 14 \\
	& & CPU & 0.0071 & 0.1584 & 0.3398 & 0.6113 & 0.9671 \\
	& & RES & 9.1454e-06 & 6.1460e-06 & 8.4822e-06 & 2.5302e-06 & 3.0251e-06 \\
		
		\addlinespace
		& \multirow{3}{*}{\shortstack{\textbf{RSMGS}\\[-0.5mm]$(\phi_1,\phi_2)=(-1,1.75)$}}
		& IT  & 10 & 10 & 10 & 10 & 10 \\
		& & CPU & 0.0107 & 0.1172 & 0.2639 & 0.3847 & 0.5870 \\
		& & RES & 2.6642e-06 & 3.1133e-06 & 3.4233e-06 & 3.7393e-06 & 4.0598e-06 \\
		
		\bottomrule
	\end{tabular*}
\end{table}

\begin{table}[htbp]
	\centering
	\caption{Performance comparison for Example~\ref{Ex2}}
	\label{tab:ex2}
	\scriptsize
	\renewcommand{\arraystretch}{1.02}
	\setlength{\tabcolsep}{2pt}
	\begin{tabular*}{\textwidth}{@{\extracolsep{\fill}}l@{}clccccc@{}}
		\toprule
		\multirow{2}{*}{} & \multirow{2}{*}{Method} & \multirow{2}{*}{}
		& \multicolumn{5}{c}{$m$} \\
		\cmidrule(lr){4-8}
		& & & 100 & 400 & 600 & 800 & 1000 \\
		\midrule
		
		\multirow{24}{*}{$\mu=4$}
		& \multirow{3}{*}{MGS}
		& IT  & 27 & 28 & 29 & 29 & 29 \\
		& & CPU & 0.0210 & 0.3530 & 0.7213 & 1.5916 & 2.1377 \\
		& & RES & 7.3855e-06 & 8.8088e-06 & 6.3440e-06 & 7.3325e-06 & 8.2027e-06 \\
		
		\addlinespace
		& \multirow{3}{*}{\shortstack{MSOR\\[-0.5mm]$\alpha=0.88$}}
		& IT  & 15 & 16 & 16 & 16 & 17 \\
		& & CPU & 0.0124 & 0.2562 & 0.4651 & 0.7413 & 1.2337 \\
		& & RES & 6.3443e-06 & 5.6451e-06 & 7.6809e-06 & 9.6710e-06 & 3.7163e-06 \\
		
		\addlinespace
		& \multirow{3}{*}{MB-DS$_L$}
		& IT  & 11 & 12 & 12 & 12 & 12 \\
		& & CPU & 0.2588 & 4.5046 & 9.9763 & 17.6842 & 60.3513 \\
		& & RES & 3.8862e-06 & 3.2320e-06 & 4.8704e-06 & 6.5088e-06 & 8.1472e-06 \\
		
		\addlinespace
	& \multirow{3}{*}{\shortstack{NMGS\\[-0.5mm]}}
	& IT  & 19 & 20 & 21 & 21 & 21 \\
	& & CPU & 0.0160 & 0.1016 & 0.2285 & 0.3974 & 0.6099 \\
	& & RES & 4.8133e-06 & 7.8671e-06 & 4.5680e-06 & 6.1160e-06 & 7.6556e-06 \\
		
		\addlinespace
		& \multirow{3}{*}{NPGS}
		& IT  & 23 & 24 & 25 & 25 & 26 \\
		& & CPU & 0.0085 & 0.1840 & 0.3798 & 0.6628 & 1.0779 \\
		& & RES & 4.7820e-06 & 9.4259e-06 & 6.6362e-06 & 8.8761e-06 & 5.1863e-06 \\
		
		\addlinespace
		& \multirow{3}{*}{\shortstack{NPSOR\\[-0.5mm]$\alpha=1.8$}}
		& IT  & 14 & 15 & 15 & 15 & 15 \\
		& & CPU & 0.0052 & 0.1104 & 0.2384 & 0.5007 & 0.6189 \\
		& & RES & 5.0434e-06 & 4.4308e-06 & 6.0616e-06 & 7.6608e-06 & 9.2448e-06 \\
		
		\addlinespace
& \multirow{3}{*}{\shortstack{RGTMSOR\\[-0.5mm]$\alpha=0.9$\\[-0.5mm]$\theta=1.2,\omega=0.1$}}
& IT  & 11 & 12 & 12 & 12 & 12 \\
& & CPU & 0.0069 & 0.1432 & 0.3087 & 0.5252 & 0.8181 \\
& & RES & 6.6125e-06 & 2.4692e-06 & 3.0875e-06 & 3.6396e-06 & 4.1520e-06 \\
		
		\addlinespace
		& \multirow{3}{*}{\shortstack{\textbf{RSMGS}\\[-0.5mm]$(\phi_1,\phi_2)=(-0.75,1.25)$}}
		& IT  & 8 & 8 & 8 & 8 & 8 \\
		& & CPU & 0.0047 & 0.0797 & 0.1647 & 0.3052 & 0.4539 \\
		& & RES & 1.2616e-06 & 1.9999e-06 & 2.3760e-06 & 2.7062e-06 & 3.0057e-06 \\
		
		\midrule
		
		\multirow{24}{*}{$\mu=6$}
		& \multirow{3}{*}{MGS}
		& IT  & 24 & 25 & 25 & 26 & 26 \\
		& & CPU & 0.0182 & 0.3196 & 0.6731 & 1.2200 & 1.8791 \\
		& & RES & 7.0551e-06 & 7.7321e-06 & 9.4857e-06 & 5.9261e-06 & 6.6290e-06 \\
		
		\addlinespace
		& \multirow{3}{*}{\shortstack{MSOR\\[-0.5mm]$\alpha=0.88$}}
		& IT  & 13 & 14 & 14 & 14 & 15 \\
		& & CPU & 0.0093 & 0.1709 & 0.3890 & 0.6646 & 1.0708 \\
		& & RES & 9.7156e-06 & 6.4264e-06 & 7.8852e-06 & 9.1144e-06 & 3.3358e-06 \\
		
		\addlinespace
		& \multirow{3}{*}{MB-DS$_L$}
		& IT  & 9 & 10 & 10 & 10 & 11 \\
		& & CPU & 0.2162 & 3.6861 & 8.2891 & 14.5716 & 24.9054 \\
		& & RES & 7.1206e-06 & 4.1990e-06 & 6.3231e-06 & 8.4472e-06 & 1.5096e-06 \\
		
		\addlinespace
	& \multirow{3}{*}{\shortstack{NMGS\\[-0.5mm]}}
	& IT  & 15 & 16 & 17 & 17 & 17 \\
	& & CPU & 0.0171 & 0.0816 & 0.1811 & 0.3096 & 0.4819 \\
	& & RES & 5.9930e-06 & 7.4016e-06 & 3.2821e-06 & 4.3888e-06 & 5.4956e-06 \\
		
		\addlinespace
		& \multirow{3}{*}{NPGS}
		& IT  & 18 & 19 & 20 & 20 & 20 \\
		& & CPU & 0.0066 & 0.1328 & 0.2984 & 0.5329 & 0.8174 \\
		& & RES & 5.0586e-06 & 7.7854e-06 & 4.3222e-06 & 5.7778e-06 & 7.2334e-06 \\
		
		\addlinespace
		& \multirow{3}{*}{\shortstack{NPSOR\\[-0.5mm]$\alpha=1.45$}}
		& IT  & 12 & 13 & 13 & 14 & 14 \\
		& & CPU & 0.0044 & 0.0925 & 0.1941 & 0.3761 & 0.5742 \\
		& & RES & 6.2229e-06 & 6.2987e-06 & 9.4828e-06 & 3.1045e-06 & 3.8850e-06 \\
		
		\addlinespace
& \multirow{3}{*}{\shortstack{RGTMSOR\\[-0.5mm]$\alpha=0.9$\\[-0.5mm]$\theta=1.1,\omega=0.1$}}
& IT  & 10 & 11 & 12 & 12 & 12 \\
& & CPU & 0.0058 & 0.1326 & 0.2974 & 0.5322 & 0.8103 \\
& & RES & 4.5899e-06 & 7.5099e-06 & 9.7055e-07 & 1.2772e-06 & 1.5837e-06 \\
		
		\addlinespace
		& \multirow{3}{*}{\shortstack{\textbf{RSMGS}\\[-0.5mm]$(\phi_1,\phi_2)=(-0.75,1.25)$}}
		& IT  & 7 & 7 & 7 & 7 & 7 \\
		& & CPU & 0.0046 & 0.0736 & 0.1722 & 0.2556 & 0.3847 \\
		& & RES & 2.7003e-06 & 4.9824e-06 & 6.0659e-06 & 6.9994e-06 & 7.8368e-06 \\
		
		\bottomrule
	\end{tabular*}
\end{table}

\begin{table}[htbp]
	\centering
	\caption{Performance comparison for Example~\ref{Ex3}}
	\label{tab:ex3}
	\scriptsize
	\renewcommand{\arraystretch}{1.02}
	\setlength{\tabcolsep}{2pt}
	\begin{tabular*}{\textwidth}{@{\extracolsep{\fill}}l@{}clccccc@{}}
		\toprule
		\multirow{2}{*}{} & \multirow{2}{*}{Method} & \multirow{2}{*}{}
		& \multicolumn{5}{c}{$m$} \\
		\cmidrule(lr){4-8}
		& & & 100 & 400 & 600 & 800 & 1000 \\
		\midrule
		
		\multirow{24}{*}{$\mu=4$}
		& \multirow{3}{*}{MGS}
		& IT  & 44 & 47 & 47 & 48 & 48 \\
		& & CPU & 0.0448 & 0.5396 & 1.1650 & 2.1731 & 3.1117 \\
		& & RES & 9.7027e-06 & 7.3058e-06 & 8.9884e-06 & 7.4139e-06 & 8.3004e-06 \\
		
		\addlinespace
		& \multirow{3}{*}{\shortstack{MSOR\\[-0.5mm]$\alpha=0.88$}}
		& IT  & 20 & 21 & 22 & 22 & 22 \\
		& & CPU & 0.0153 & 0.2357 & 0.5566 & 1.0045 & 1.4873 \\
		& & RES & 6.8065e-06 & 7.9965e-06 & 4.6639e-06 & 5.8673e-06 & 7.0564e-06 \\
		
		\addlinespace
		& \multirow{3}{*}{MB-DS$_L$}
		& IT  & 20 & 22 & 22 & 23 & 23 \\
		& & CPU & 0.2546 & 3.9438 & 8.7513 & 16.3213 & 24.3193 \\
		& & RES & 7.5412e-06 & 6.1841e-06 & 9.3951e-06 & 5.3980e-06 & 6.7740e-06 \\
		
		\addlinespace
	& \multirow{3}{*}{\shortstack{NMGS\\[-0.5mm]}}
	& IT  & 23 & 25 & 26 & 26 & 26 \\
	& & CPU & 0.0171 & 0.1229 & 0.2625 & 0.4675 & 0.7130 \\
	& & RES & 8.3765e-06 & 8.1792e-06 & 5.7946e-06 & 7.7775e-06 & 9.7603e-06 \\
		
		\addlinespace
		& \multirow{3}{*}{NPGS}
		& IT  & 27 & 29 & 30 & 31 & 31 \\
		& & CPU & 0.0095 & 0.2074 & 0.4942 & 0.9220 & 1.2202 \\
		& & RES & 7.6859e-06 & 9.5827e-06 & 7.6954e-06 & 5.4622e-06 & 6.8534e-06 \\
		
		\addlinespace
		& \multirow{3}{*}{\shortstack{NPSOR\\[-0.5mm]$\alpha=1.9$}}
		& IT  & 19 & 21 & 21 & 21 & 21 \\
		& & CPU & 0.0086 & 0.1776 & 0.3411 & 0.5748 & 0.8652 \\
		& & RES & 9.8799e-06 & 4.6382e-06 & 6.0085e-06 & 7.2912e-06 & 8.5259e-06 \\
		
		\addlinespace
	& \multirow{3}{*}{\shortstack{RGTMSOR\\[-0.5mm]$\alpha=0.9$\\[-0.5mm]$\theta=1.1,\omega=0.1$}}
	& IT  & 17 & 18 & 19 & 19 & 19 \\
	& & CPU & 0.0087 & 0.2138 & 0.4774 & 0.8392 & 1.2856 \\
	& & RES & 7.6449e-06 & 7.8188e-06 & 3.9091e-06 & 4.9185e-06 & 5.9178e-06 \\
		
		\addlinespace
		& \multirow{3}{*}{\shortstack{\textbf{RSMGS}\\[-0.5mm]$(\phi_1,\phi_2)=(-2,2)$}}
		& IT  & 16 & 16 & 16 & 16 & 16 \\
		& & CPU & 0.0067 & 0.1012 & 0.1892 & 0.3268 & 0.5136 \\
		& & RES & 4.5752e-06 & 5.0196e-06 & 5.3690e-06 & 5.7948e-06 & 6.2814e-06 \\
		
		\midrule
		
		\multirow{24}{*}{$\mu=6$}
		& \multirow{3}{*}{MGS}
		& IT  & 34 & 36 & 36 & 37 & 37 \\
		& & CPU & 0.0235 & 0.4034 & 0.8712 & 1.6793 & 2.5007 \\
		& & RES & 9.4515e-06 & 7.9517e-06 & 9.7702e-06 & 7.2283e-06 & 8.0894e-06 \\
		
		\addlinespace
		& \multirow{3}{*}{\shortstack{MSOR\\[-0.5mm]$\alpha=0.85$}}
		& IT  & 16 & 17 & 18 & 18 & 18 \\
		& & CPU & 0.0106 & 0.1939 & 0.4376 & 0.8380 & 1.2808 \\
		& & RES & 6.9276e-06 & 9.7397e-06 & 4.7390e-06 & 6.3517e-06 & 7.9644e-06 \\
		
		\addlinespace
		& \multirow{3}{*}{MB-DS$_L$}
		& IT  & 16 & 17 & 18 & 18 & 18 \\
		& & CPU & 0.2435 & 3.8297 & 8.3767 & 41.9923 & 49.3077 \\
		& & RES & 6.1965e-06 & 9.0231e-06 & 4.5513e-06 & 6.1001e-06 & 7.6489e-06 \\
		
		\addlinespace
	& \multirow{3}{*}{\shortstack{NMGS\\[-0.5mm]}}
	& IT  & 18 & 20 & 20 & 21 & 21 \\
	& & CPU & 0.0138 & 0.0995 & 0.2102 & 0.3824 & 0.6056 \\
	& & RES & 8.3138e-06 & 4.9369e-06 & 7.4843e-06 & 3.6933e-06 & 4.6318e-06 \\
		
		\addlinespace
		& \multirow{3}{*}{NPGS}
		& IT  & 21 & 23 & 23 & 24 & 24 \\
		& & CPU & 0.0074 & 0.1750 & 0.3344 & 0.7085 & 1.0303 \\
		& & RES & 6.6945e-06 & 5.3469e-06 & 8.0989e-06 & 4.6478e-06 & 5.8272e-06 \\
		
		\addlinespace
		& \multirow{3}{*}{\shortstack{NPSOR\\[-0.5mm]$\alpha=1.6$}}
		& IT  & 15 & 16 & 17 & 17 & 17 \\
		& & CPU & 0.0061 & 0.1209 & 0.2451 & 0.5163 & 0.7173 \\
		& & RES & 7.7009e-06 & 7.2529e-06 & 3.2936e-06 & 4.1657e-06 & 5.0297e-06 \\
		
		\addlinespace
& \multirow{3}{*}{\shortstack{RGTMSOR\\[-0.5mm]$\alpha=0.9$\\[-0.5mm]$\theta=1,\omega=0.1$}}
& IT  & 14 & 15 & 16 & 16 & 16 \\
& & CPU & 0.0098 & 0.1748 & 0.3946 & 0.7049 & 1.1197 \\
& & RES & 7.3573e-06 & 8.2839e-06 & 3.3340e-06 & 4.4527e-06 & 5.5715e-06 \\
		
		\addlinespace
		& \multirow{3}{*}{\shortstack{\textbf{RSMGS}\\[-0.5mm]$(\phi_1,\phi_2)=(-2,2.25)$}}
		& IT  & 13 & 13 & 13 & 13 & 13 \\
		& & CPU & 0.0061 & 0.0840 & 0.1603 & 0.2668 & 0.4179 \\
		& & RES & 3.5083e-06 & 3.8978e-06 & 4.1840e-06 & 4.4866e-06 & 4.8026e-06 \\
		
		\bottomrule
	\end{tabular*}
\end{table}

\begin{table}[htbp]
	\centering
	\caption{Performance comparison for Example~\ref{Ex4}}
	\label{tab:ex4}
	\scriptsize
	\renewcommand{\arraystretch}{1.02}
	\setlength{\tabcolsep}{2pt}
	\begin{tabular*}{\textwidth}{@{\extracolsep{\fill}}l@{}clccccc@{}}
		\toprule
		\multirow{2}{*}{} & \multirow{2}{*}{Method} & \multirow{2}{*}{}
		& \multicolumn{5}{c}{$m$} \\
		\cmidrule(lr){4-8}
		& & & 100 & 400 & 600 & 800 & 1000 \\
		\midrule
		
		\multirow{24}{*}{$\mu=4$}
		& \multirow{3}{*}{MGS}
		& IT  & 19 & 20 & 20 & 21 & 21 \\
		& & CPU & 0.0193 & 0.2686 & 0.5828 & 0.8619 & 1.5133 \\
		& & RES & 7.2507e-06 & 7.1261e-06 & 8.6953e-06 & 5.0709e-06 & 5.6623e-06 \\
		
		\addlinespace
		& \multirow{3}{*}{\shortstack{MSOR\\[-0.5mm]$\alpha=0.80$}}
		& IT  & 11 & 13 & 14 & 14 & 14 \\
		& & CPU & 0.0087 & 0.1841 & 0.3421 & 0.7373 & 1.0821 \\
		& & RES & 7.5268e-06 & 7.0841e-06 & 2.8860e-06 & 3.8510e-06 & 4.8160e-06 \\
		
		\addlinespace
		& \multirow{3}{*}{MB-DS$_L$}
		& IT  & 10 & 11 & 11 & 11 & 11 \\
		& & CPU & 2.0121 & 4.1712 & 9.1291 & 16.1363 & 56.2591 \\
		& & RES & 4.9842e-06 & 2.9620e-06 & 4.4640e-06 & 5.9659e-06 & 7.4679e-06 \\
		
		\addlinespace
	& \multirow{3}{*}{\shortstack{NMGS\\[-0.5mm]}}
	& IT  & 10 & 11 & 11 & 11 & 11 \\
	& & CPU & 0.0090 & 0.0548 & 0.1263 & 0.2030 & 0.3072 \\
	& & RES & 4.1746e-06 & 2.3883e-06 & 3.5699e-06 & 4.7512e-06 & 5.9325e-06 \\
		
		\addlinespace
		& \multirow{3}{*}{NPGS}
		& IT  & 102 & 113 & 116 & 118 & 120 \\
		& & CPU & 0.0423 & 0.8268 & 1.7710 & 3.1464 & 5.0722 \\
		& & RES & 9.4194e-06 & 8.8765e-06 & 9.0235e-06 & 9.2886e-06 & 8.9716e-06 \\
		
		\addlinespace
		& \multirow{3}{*}{\shortstack{NPSOR\\[-0.5mm]$\alpha=1.95$}}
		& IT  & 75 & 82 & 84 & 86 & 87 \\
		& & CPU & 0.0301 & 0.6226 & 1.2575 & 2.2686 & 3.5721 \\
		& & RES & 8.9084e-06 & 9.1470e-06 & 9.5347e-06 & 8.9286e-06 & 9.3398e-06 \\
		
		\addlinespace
& \multirow{3}{*}{\shortstack{RGTMSOR\\[-0.5mm]$\alpha=1$\\[-0.5mm]$\theta=0.8,\omega=0.1$}}
& IT  & 10 & 11 & 11 & 12 & 12 \\
& & CPU & 0.0052 & 0.1336 & 0.2764 & 0.5334 & 0.8012 \\
& & RES & 7.7785e-06 & 6.1583e-06 & 8.5545e-06 & 2.8908e-06 & 3.4395e-06 \\
		
		\addlinespace
		& \multirow{3}{*}{\shortstack{\textbf{RSMGS}\\[-0.5mm]$(\phi_1,\phi_2)=(-1,0)$}}
		& IT  & 9 & 9 & 9 & 9 & 10 \\
		& & CPU & 0.0050 & 0.0602 & 0.1151 & 0.1916 & 0.3514 \\
		& & RES & 4.4875e-06 & 7.0999e-06 & 8.4380e-06 & 9.6167e-06 & 1.7865e-06 \\
		
		\midrule
		
		\multirow{24}{*}{$\mu=6$}
		& \multirow{3}{*}{MGS}
		& IT  & 17 & 18 & 19 & 19 & 19 \\
		& & CPU & 0.0132 & 0.2561 & 0.4447 & 0.8087 & 1.4010 \\
		& & RES & 9.8602e-06 & 9.1619e-06 & 5.2998e-06 & 6.1124e-06 & 6.8290e-06 \\
		
		\addlinespace
		& \multirow{3}{*}{\shortstack{MSOR\\[-0.5mm]$\alpha=0.85$}}
		& IT  & 10 & 11 & 11 & 11 & 11 \\
		& & CPU & 0.0083 & 0.1561 & 0.2661 & 0.5741 & 0.9419 \\
		& & RES & 7.4848e-06 & 3.4084e-06 & 4.6944e-06 & 5.9604e-06 & 7.2169e-06 \\
		
		\addlinespace
		& \multirow{3}{*}{MB-DS$_L$}
		& IT  & 9 & 10 & 10 & 10 & 10 \\
		& & CPU & 0.2168 & 3.6420 & 8.2379 & 14.5351 & 22.8988 \\
		& & RES & 4.6920e-06 & 2.1617e-06 & 3.2568e-06 & 4.3519e-06 & 5.4470e-06 \\
		
		\addlinespace
	& \multirow{3}{*}{\shortstack{NMGS\\[-0.5mm]}}
	& IT  & 9 & 10 & 10 & 10 & 10 \\
	& & CPU & 0.0068 & 0.0517 & 0.1127 & 0.1821 & 0.2794 \\
	& & RES & 4.1021e-06 & 1.8372e-06 & 2.7505e-06 & 3.6637e-06 & 4.5769e-06 \\
		
		\addlinespace
		& \multirow{3}{*}{NPGS}
		& IT  & 121 & 133 & 137 & 140 & 142 \\
		& & CPU & 0.0511 & 0.9339 & 2.0593 & 3.6426 & 6.0710 \\
		& & RES & 9.4775e-06 & 9.7680e-06 & 9.5887e-06 & 9.3417e-06 & 9.4777e-06 \\
		
		\addlinespace
		& \multirow{3}{*}{\shortstack{NPSOR\\[-0.5mm]$\alpha=1.95$}}
		& IT  & 90 & 95 & 97 & 98 & 99 \\
		& & CPU & 0.0409 & 0.6756 & 1.4310 & 2.8074 & 4.0826 \\
		& & RES & 8.8453e-06 & 9.3737e-06 & 9.0174e-06 & 9.3402e-06 & 9.3572e-06 \\
		
		\addlinespace
& \multirow{3}{*}{\shortstack{RGTMSOR\\[-0.5mm]$\alpha=0.9$\\[-0.5mm]$\theta=1,\omega=0.1$}}
& IT  & 10 & 11 & 11 & 11 & 11 \\
& & CPU & 0.0069 & 0.1351 & 0.2799 & 0.4884 & 0.7466 \\
& & RES & 5.5292e-06 & 3.7595e-06 & 4.6177e-06 & 5.3466e-06 & 5.9937e-06 \\
		
		\addlinespace
		& \multirow{3}{*}{\shortstack{\textbf{RSMGS}\\[-0.5mm]$(\phi_1,\phi_2)=(-1,0)$}}
		& IT  & 8 & 8 & 9 & 9 & 9 \\
		& & CPU & 0.0049 & 0.0555 & 0.1279 & 0.2203 & 0.3412 \\
		& & RES & 4.6782e-06 & 8.3383e-06 & 1.3128e-06 & 1.4909e-06 & 1.6519e-06 \\
		
		\bottomrule
	\end{tabular*}
\end{table}

\subsection{Selection of Efficient Parameters \texorpdfstring{$\phi_1$ and $\phi_2$}{phi1 and phi2}}
\label{ss5.2}

The convergence interval in Theorem~\ref{thm:parameter-interval} provides a
sufficient region of admissible parameters, that is, parameter values for which
the convergence of the proposed method is guaranteed. However, this condition is
only sufficient and does not necessarily identify the fastest convergent
parameters. Different admissible pairs may lead to different iteration counts,
CPU times, and final residuals.

Therefore, the grid-search procedure in Algorithm~\ref{algo2} is used to select
efficient parameters for the numerical experiments. In the computations, the
grid step sizes are taken as \(h_1=h_2=0.25\). The selected pair
\((\phi_1,\phi_2)\) is the one that gives the smallest number of iterations over
the prescribed grid. If several pairs give the same minimum iteration count,
the pair with the smallest final residual is chosen.

For \(\mu=4\) and \(\mu=6\), the admissible interval obtained from the
sufficient convergence theorem is used as the initial search region in
Algorithm~\ref{algo2}. For each example, the pair \((\phi_1,\phi_2)\) is
selected at the fixed problem size \(m=1000\), and the same pair is then kept
fixed for all other values of \(m\).

\begin{algorithm}[h]
	\caption{Grid-search procedure for selecting \((\phi_1,\phi_2)\)}
	\label{algo2}
	\begin{algorithmic}[1]
		\Require Matrix \(A\), vector \(q\), initial vector \(x^{(0)}\), tolerance
		\(\varepsilon\), maximum iteration number \(K_{\max}\), fixed \(m\), fixed
		\(\mu\), and step sizes \(h_1,h_2>0\).
		\Ensure Selected parameter pair \((\phi_1,\phi_2)\).
		
		\State Choose search intervals
		\(\phi_1\in[\phi_1^{\ell},\phi_1^{u}]\) and
		\(\phi_2\in[\phi_2^{\ell},\phi_2^{u}]\).
		
		\State Construct the finite grids
		\begin{equation*}
			\Phi_1
			=
			\{\phi_1^{\ell},\phi_1^{\ell}+h_1,\ldots,\phi_1^{u}\},
			\qquad
			\Phi_2
			=
			\{\phi_2^{\ell},\phi_2^{\ell}+h_2,\ldots,\phi_2^{u}\}.
		\end{equation*}
		
		\State Set \(IT_{\min}=+\infty\), \(RES_{\min}=+\infty\), and initialize
		\((\phi_1,\phi_2)\).
		
		\For{each \(\widehat{\phi}_1\in\Phi_1\)}
		\For{each \(\widehat{\phi}_2\in\Phi_2\)}
		\State Run RSMGS with \((\widehat{\phi}_1,\widehat{\phi}_2)\).
		\State Compute \(IT(\widehat{\phi}_1,\widehat{\phi}_2)\) and
		\(RES(\widehat{\phi}_1,\widehat{\phi}_2)\).
		
		\If{the method converges and
			\(RES(\widehat{\phi}_1,\widehat{\phi}_2)\leq\varepsilon\)}
		\If{\(IT(\widehat{\phi}_1,\widehat{\phi}_2)<IT_{\min}\)}
		\State Set
		\begin{equation*}
			IT_{\min}=IT(\widehat{\phi}_1,\widehat{\phi}_2),
			\qquad
			RES_{\min}=RES(\widehat{\phi}_1,\widehat{\phi}_2),
			\qquad
			(\phi_1,\phi_2)=(\widehat{\phi}_1,\widehat{\phi}_2).
		\end{equation*}
		\ElsIf{\(IT(\widehat{\phi}_1,\widehat{\phi}_2)=IT_{\min}\) and
			\(RES(\widehat{\phi}_1,\widehat{\phi}_2)<RES_{\min}\)}
		\State Set
		\begin{equation*}
			RES_{\min}=RES(\widehat{\phi}_1,\widehat{\phi}_2),
			\qquad
			(\phi_1,\phi_2)=(\widehat{\phi}_1,\widehat{\phi}_2).
		\end{equation*}
		\EndIf
		\EndIf
		\EndFor
		\EndFor
		
		\State \Return \((\phi_1,\phi_2)\).
		
	\end{algorithmic}
\end{algorithm}

For all examples, the efficient parameter-selection plots for
\(\mu=4,6\) are shown in Figures~\ref{fig1}--\ref{fig4}. The selected
parameter pairs for all tested values of \(\mu\) are also reported in the
corresponding result tables.
\begin{figure}[H]
	\centering
	\begin{subfigure}{0.48\textwidth}
		\centering
		\includegraphics[width=\textwidth]{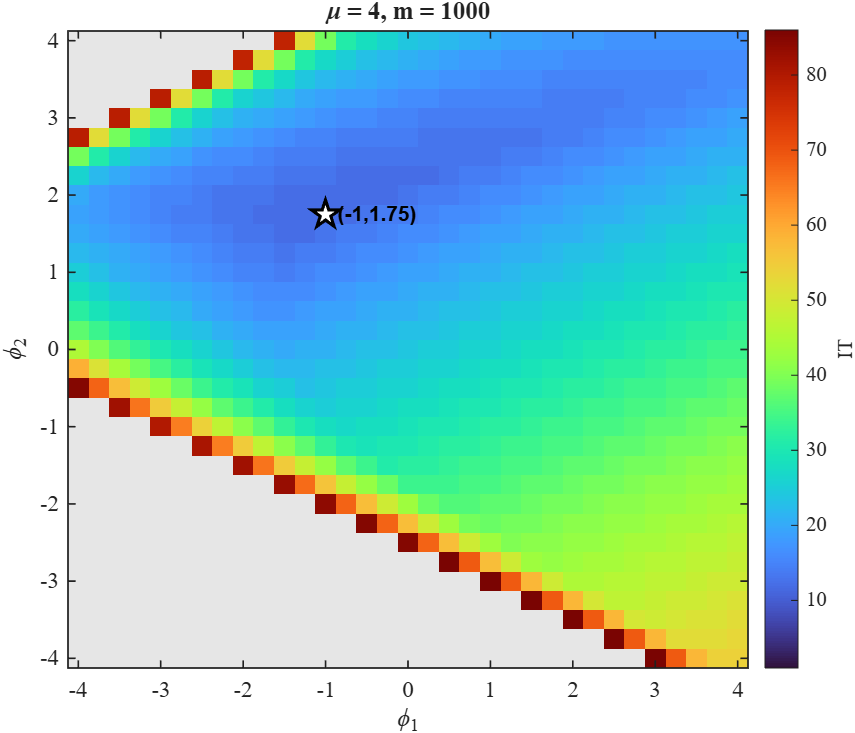}
		\caption{\(\mu=4\)}
		\label{fig:ex1-mu4}
	\end{subfigure}
	\hfill
	\begin{subfigure}{0.48\textwidth}
		\centering
		\includegraphics[width=\textwidth]{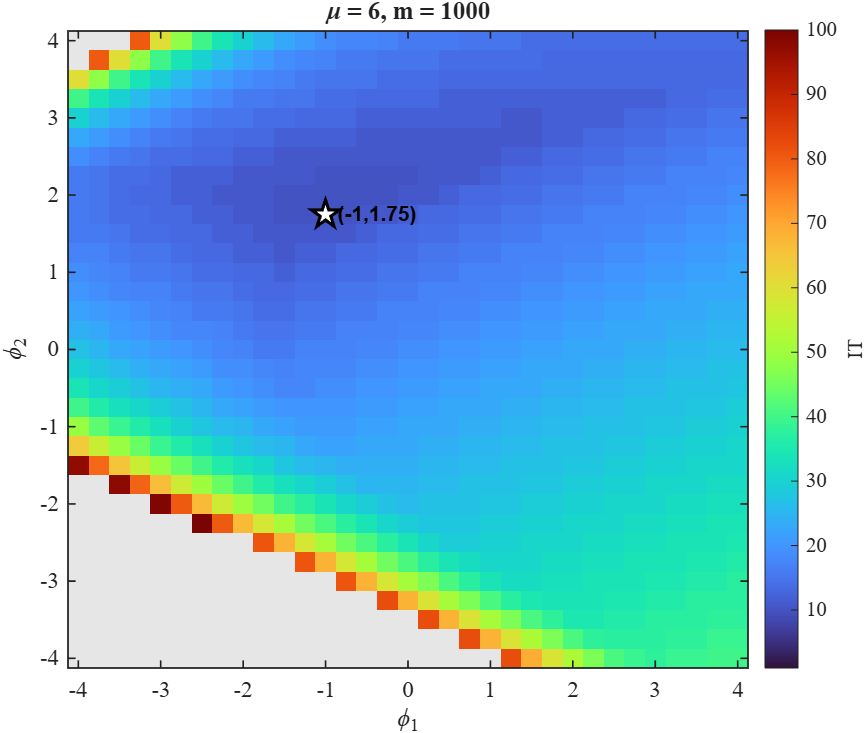}
		\caption{\(\mu=6\)}
		\label{fig:ex1-mu6}
	\end{subfigure}
	\caption{Efficient parameter-selection plots in the \((\phi_1,\phi_2)\)-plane for Example~\ref{Ex1}.}
	\label{fig1}
\end{figure}

\begin{figure}[H]
	\centering
	\begin{subfigure}{0.48\textwidth}
		\centering
		\includegraphics[width=\textwidth]{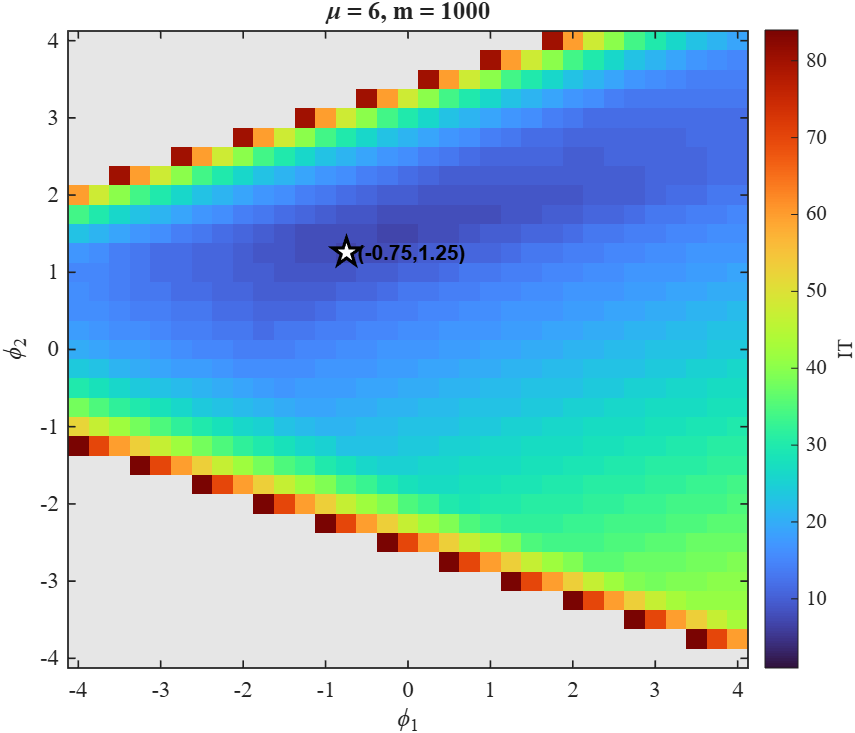}
		\caption{\(\mu=4\)}
		\label{fig:ex2-mu4}
	\end{subfigure}
	\hfill
	\begin{subfigure}{0.48\textwidth}
		\centering
		\includegraphics[width=\textwidth]{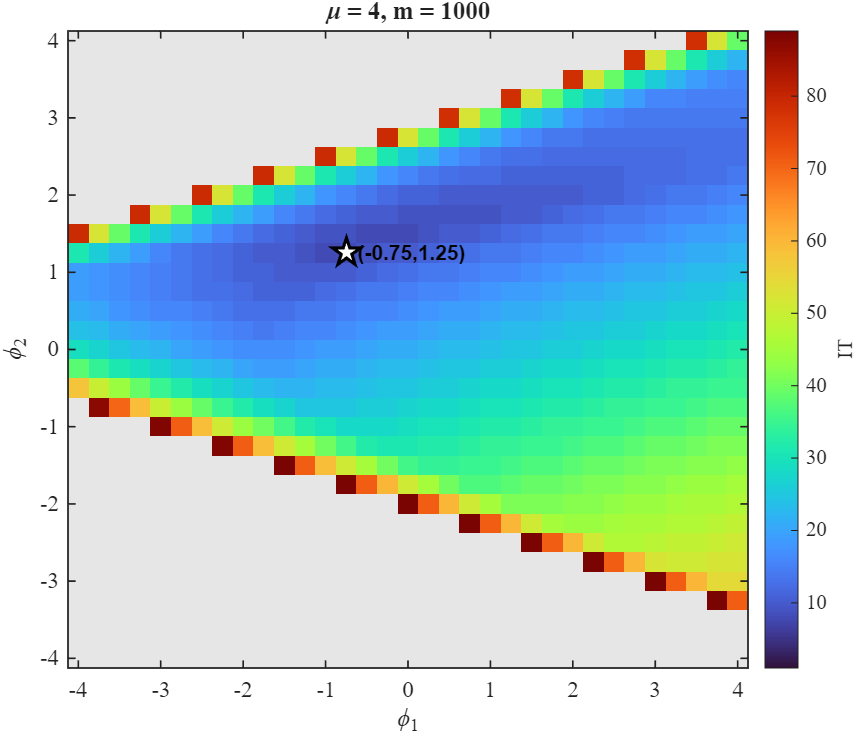}
		\caption{\(\mu=6\)}
		\label{fig:ex2-mu6}
	\end{subfigure}
	\caption{Efficient parameter-selection plots in the \((\phi_1,\phi_2)\)-plane for Example~\ref{Ex2}.}
	\label{fig2}
\end{figure}

\begin{figure}[H]
	\centering
	\begin{subfigure}{0.48\textwidth}
		\centering
		\includegraphics[width=\textwidth]{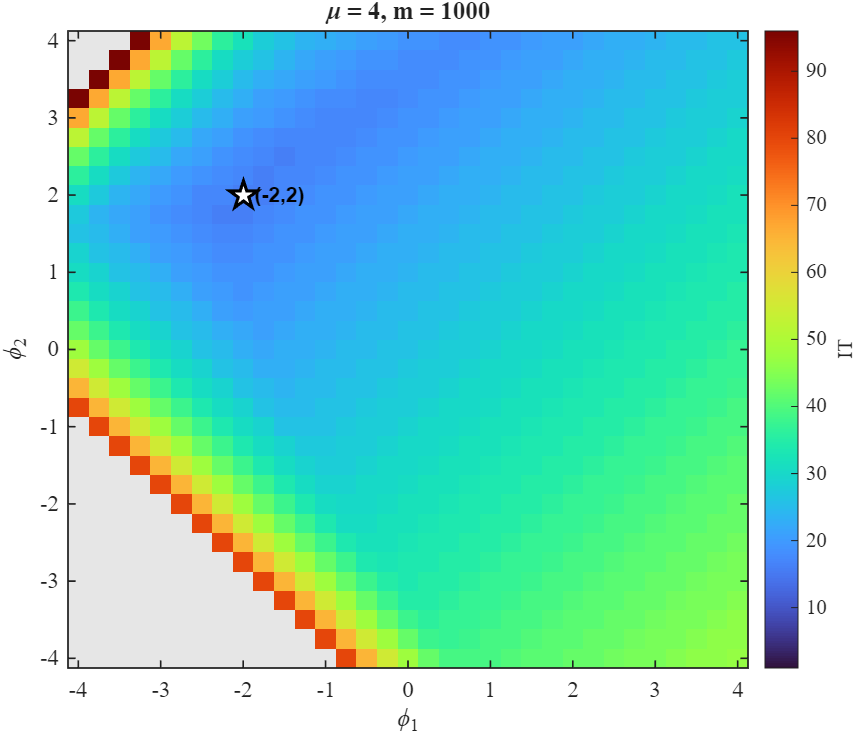}
		\caption{\(\mu=4\)}
		\label{fig:ex3-mu4}
	\end{subfigure}
	\hfill
	\begin{subfigure}{0.48\textwidth}
		\centering
		\includegraphics[width=\textwidth]{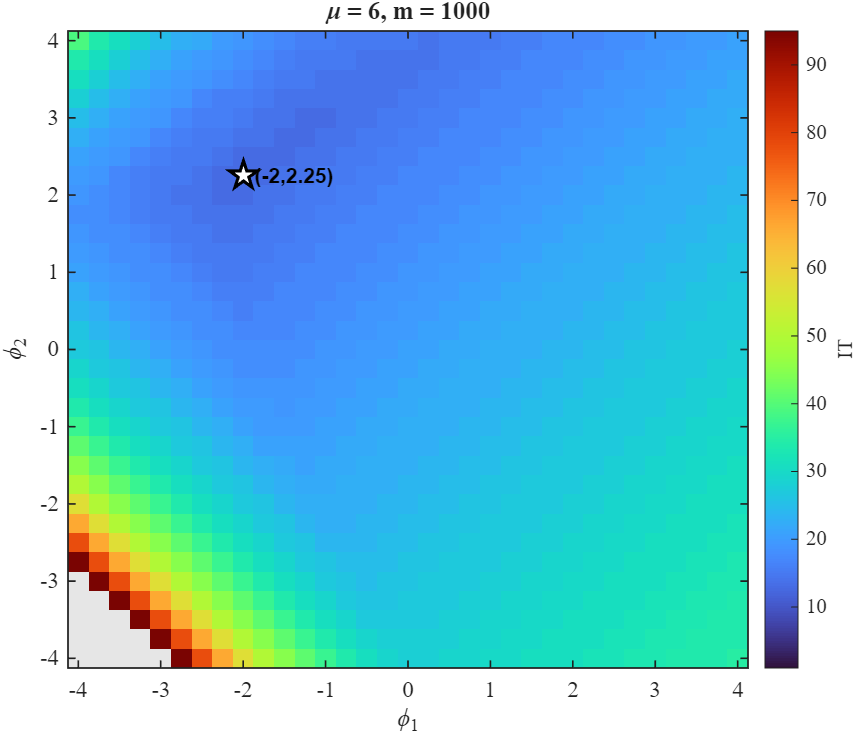}
		\caption{\(\mu=6\)}
		\label{fig:ex3-mu6}
	\end{subfigure}
	\caption{Efficient parameter-selection plots in the \((\phi_1,\phi_2)\)-plane for Example~\ref{Ex3}.}
	\label{fig3}
\end{figure}

\begin{figure}[H]
	\centering
	\begin{subfigure}{0.48\textwidth}
		\centering
		\includegraphics[width=\textwidth]{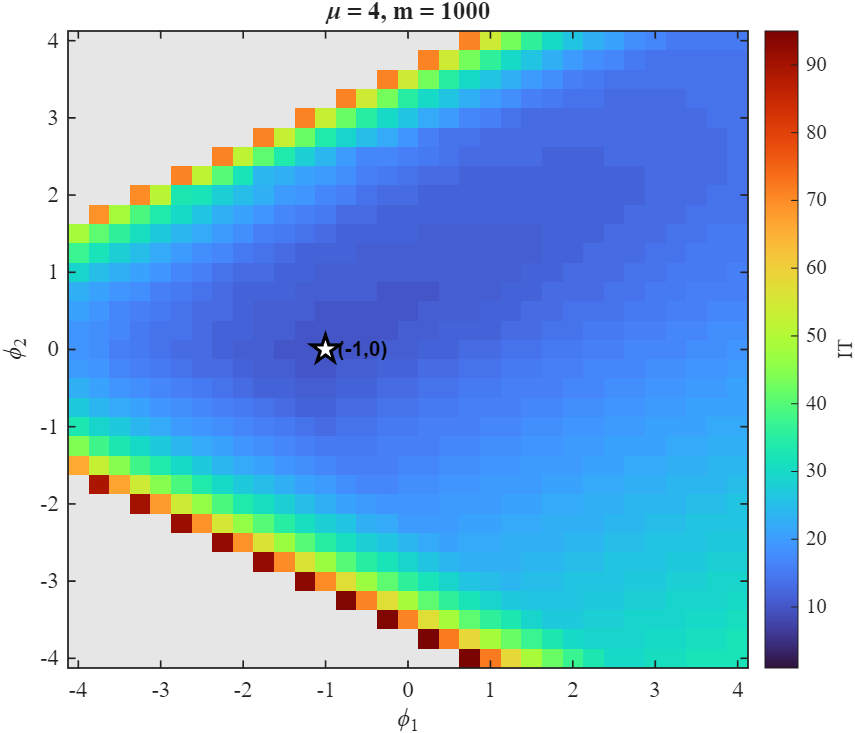}
		\caption{\(\mu=4\)}
		\label{fig:ex4-mu4}
	\end{subfigure}
	\hfill
	\begin{subfigure}{0.48\textwidth}
		\centering
		\includegraphics[width=\textwidth]{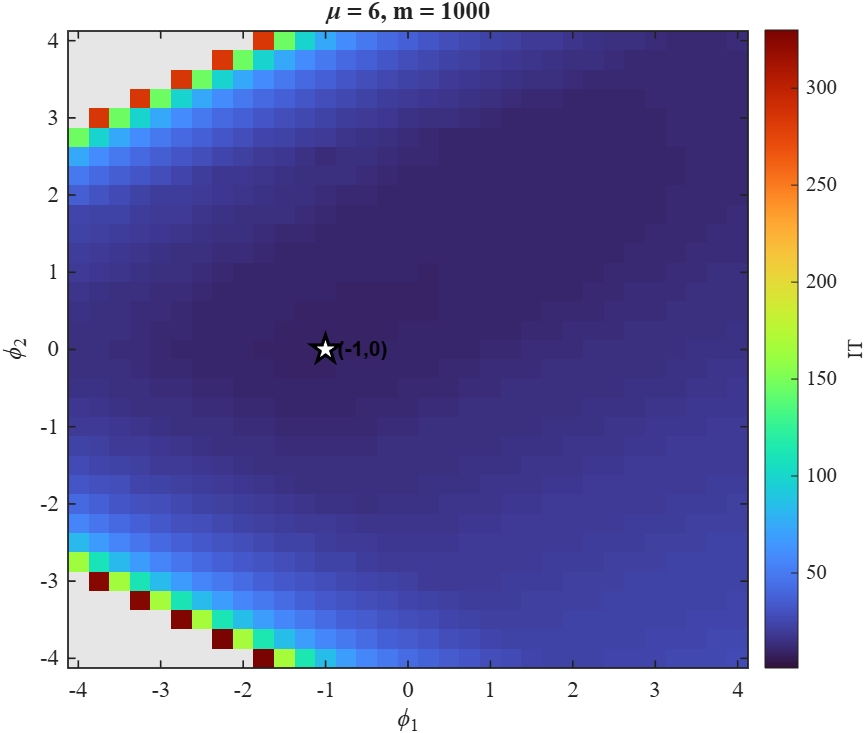}
		\caption{\(\mu=6\)}
		\label{fig:ex4-mu6}
	\end{subfigure}
	\caption{Efficient parameter-selection plots in the \((\phi_1,\phi_2)\)-plane for Example~\ref{Ex4}.}
	\label{fig4}
\end{figure}

To verify that the experimentally selected RSMGS parameters are also
covered by the sufficient convergence theory, we evaluate the scalar margin
\(\Delta_{1,1}\) from Remark~\ref{Rem4.8} at the largest tested problem size,
namely \(m=1000\). Since the values of \(\kappa_m\) for the present structured
examples are increasing with \(m\), positivity of the margin at \(m=1000\)
also guarantees positivity for the smaller values of \(m\) used in the
experiments.

\begin{table}[htbp]
	\centering
	\caption{Verification of the sufficient convergence condition for RSMGS at \(m=1000\)}
	\label{tab:rsmgs-theory-check}
	\small
	\renewcommand{\arraystretch}{1.12}
	\begin{tabular*}{\textwidth}{@{\extracolsep{\fill}}ccccc@{}}
		\toprule
		Example & \(\kappa_m\) & \(\mu\) & \((\phi_1,\phi_2)\) & \(\Delta_{1,1}\) \\
		\midrule
		\multirow{2}{*}{Example~\ref{Ex1}} & \multirow{2}{*}{4}
		& 4 & \((-1,1.75)\) & 0.7336 \\
		& & 6 & \((-1,1.75)\) & 4.7336 \\
		
		\addlinespace
		\multirow{2}{*}{Example~\ref{Ex2}} & \multirow{2}{*}{3.4641}
		& 4 & \((-0.75,1.25)\) & 4.1077 \\
		& & 6 & \((-0.75,1.25)\) & 8.1077 \\
		
		\addlinespace
		\multirow{2}{*}{Example~\ref{Ex3}} & \multirow{2}{*}{2}
		& 4 & \((-2,2)\) & 5.0718 \\
		& & 6 & \((-2,2.25)\) & 8.7889 \\
		
		\addlinespace
		\multirow{2}{*}{Example~\ref{Ex4}} & \multirow{2}{*}{4}
		& 4 & \((-1,0)\) & 6 \\
		& & 6 & \((-1,0)\) & 10 \\
		\botrule
	\end{tabular*}
\end{table}

All values of \(\Delta_{1,1}\) are positive. Hence the selected RSMGS
parameters satisfy the sufficient convergence condition for the largest
problem size considered, and therefore also for the smaller problem sizes in
Tables~\ref{tab:ex1}--\ref{tab:ex4}.

\subsection{Sensitivity Analysis for \texorpdfstring{\(\Omega_1\)}{Omega1} and \texorpdfstring{\(\Omega_2\)}{Omega2}}
In this subsection, we investigate how different choices of \(\Omega_1\) and
\(\Omega_2\) influence the numerical performance of RSMGS. We compare the following three choices:
$$
\Omega_1=\phi_1 I,\quad \Omega_2=\phi_2 L,
\qquad
\Omega_1=\phi_1 I,\quad \Omega_2=\phi_2 U,
\qquad
\Omega_1=\phi_1 I,\quad \Omega_2=\phi_2 D.
$$
These choices are not exhaustive; they are selected as representative
structured choices based on the diagonal, lower triangular, and upper
triangular parts of the splitting \(A=D-L-U\). A complete comparison of all
possible structured choices may be considered in future work.
For each choice, the efficient parameter pair \((\phi_1,\phi_2)\) is first
selected at the fixed size \(m=100\) with \(\mu=4\) using Algorithm~\ref{algo2}. The same selected pair is then kept fixed and used for all larger values of \(m\). This allows us
to check whether the selected parameters remain effective as the problem size
increases. The selected efficient pairs obtained from this tuning step are listed in
Table~\ref{tab:omega-choice-params}. 

\begin{table}[htbp]
	\centering
\caption{Efficient pairs \((\phi_1,\phi_2)\) selected at \(m=100\) and \(\mu=4\)}
	\label{tab:omega-choice-params}
	\small
	\renewcommand{\arraystretch}{1.15}
	\setlength{\tabcolsep}{3pt}
	\begin{tabular*}{\textwidth}{@{\extracolsep{\fill}}ccccc@{}}
		\toprule
		\multirow{2}{*}{\(\Omega_1,\Omega_2\)}
		& \multicolumn{4}{c}{Selected efficient pair \((\phi_1,\phi_2)\)} \\
		\cmidrule(lr){2-5}
		& Example~\ref{Ex1} & Example~\ref{Ex2} & Example~\ref{Ex3} & Example~\ref{Ex4} \\
		\midrule
		
		\(\Omega_1=\phi_1 I,\ \Omega_2=\phi_2 L\)
		& \((-1.25,1.5)\) & \((-0.75,1.25)\) & \((-2,1.75)\) & \((-0.75,0.25)\) \\
		
		\(\Omega_1=\phi_1 I,\ \Omega_2=\phi_2 U\)
		& \((-0.75,2.25)\) & \((-1,2.75)\) & \((0.25,2.25)\) & \((-0.5,1)\) \\
		
		\(\Omega_1=\phi_1 I,\ \Omega_2=\phi_2 D\)
		& \((-3.25,-0.25)\) & \((-3.75,-0.25)\) & \((-3.5,-0.25)\) & \((-3,-0.25)\) \\
		
		\botrule
	\end{tabular*}
\end{table}

\begin{figure}[h]
	\centering
	\setlength{\tabcolsep}{1pt}
	\begin{tabular}{@{}cc@{}}
		\includegraphics[width=0.50\textwidth]{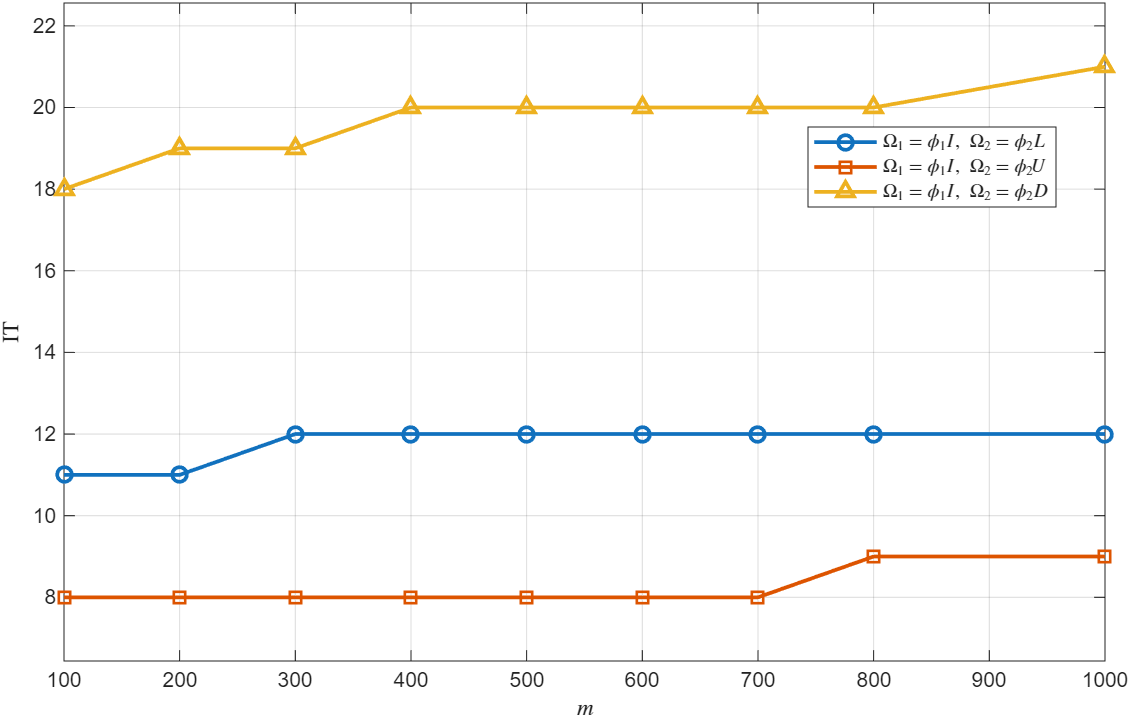}
		&
		\includegraphics[width=0.50\textwidth]{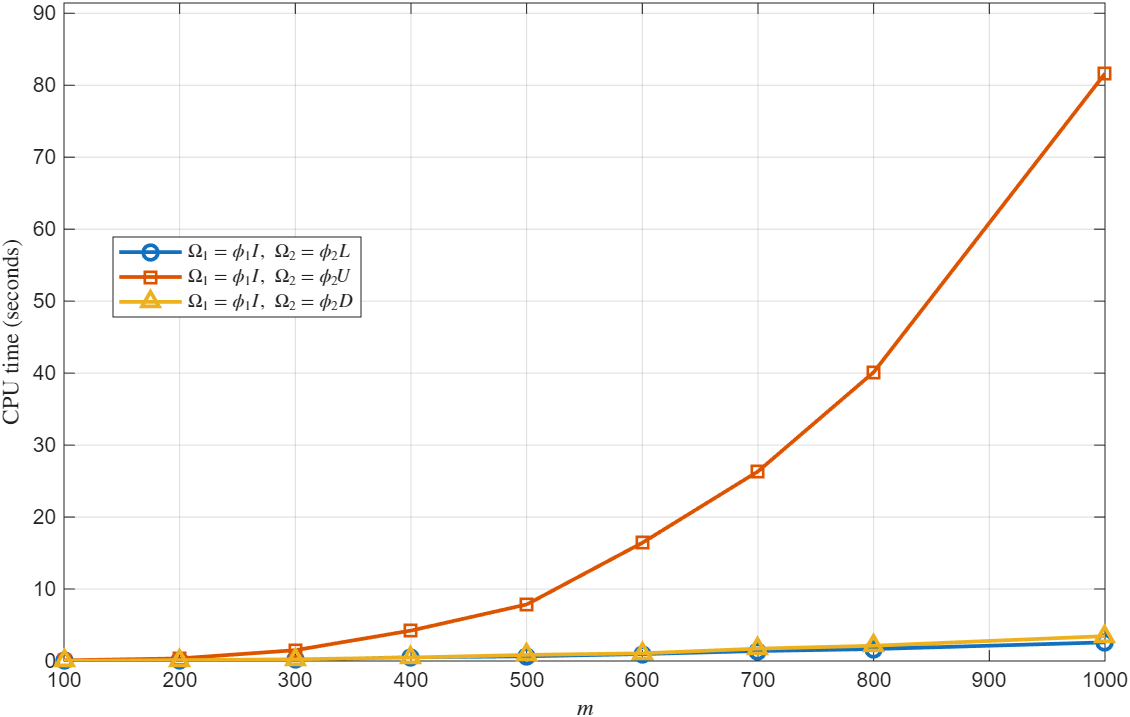}
	\end{tabular}
\caption{Sensitivity results for Example~\ref{Ex1}.}
	\label{fig:sensitivity-ex1}
\end{figure}

\begin{figure}[h]
	\centering
	\setlength{\tabcolsep}{1pt}
	\begin{tabular}{@{}cc@{}}
		\includegraphics[width=0.50\textwidth]{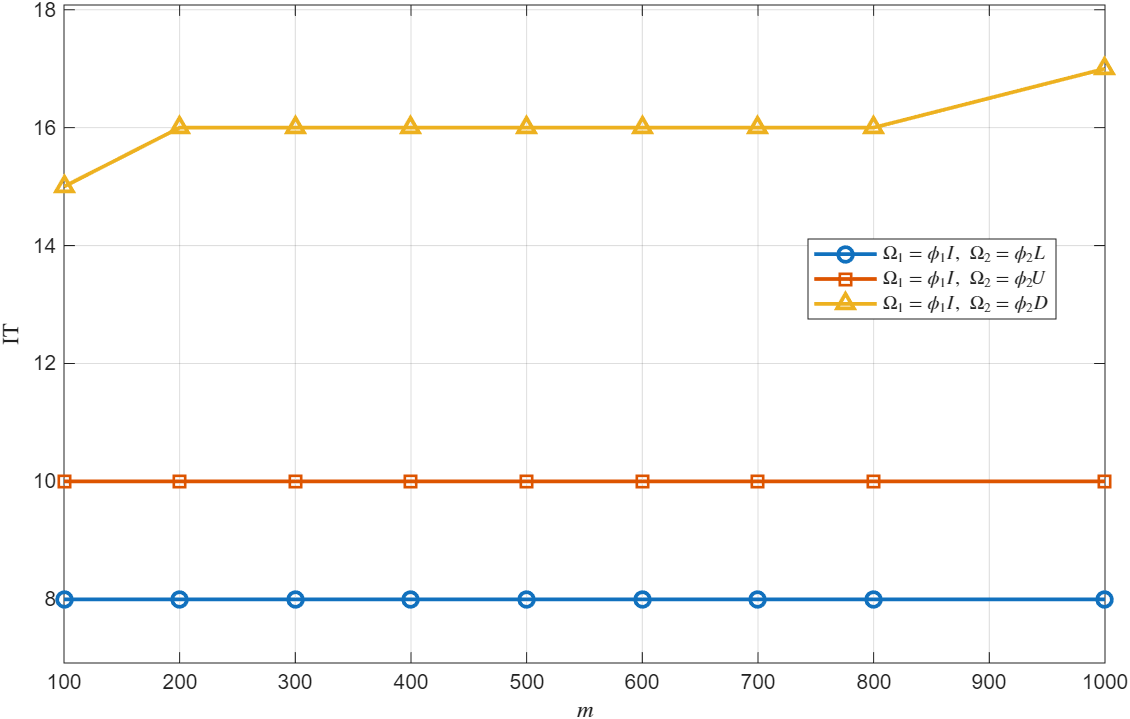}
		&
		\includegraphics[width=0.50\textwidth]{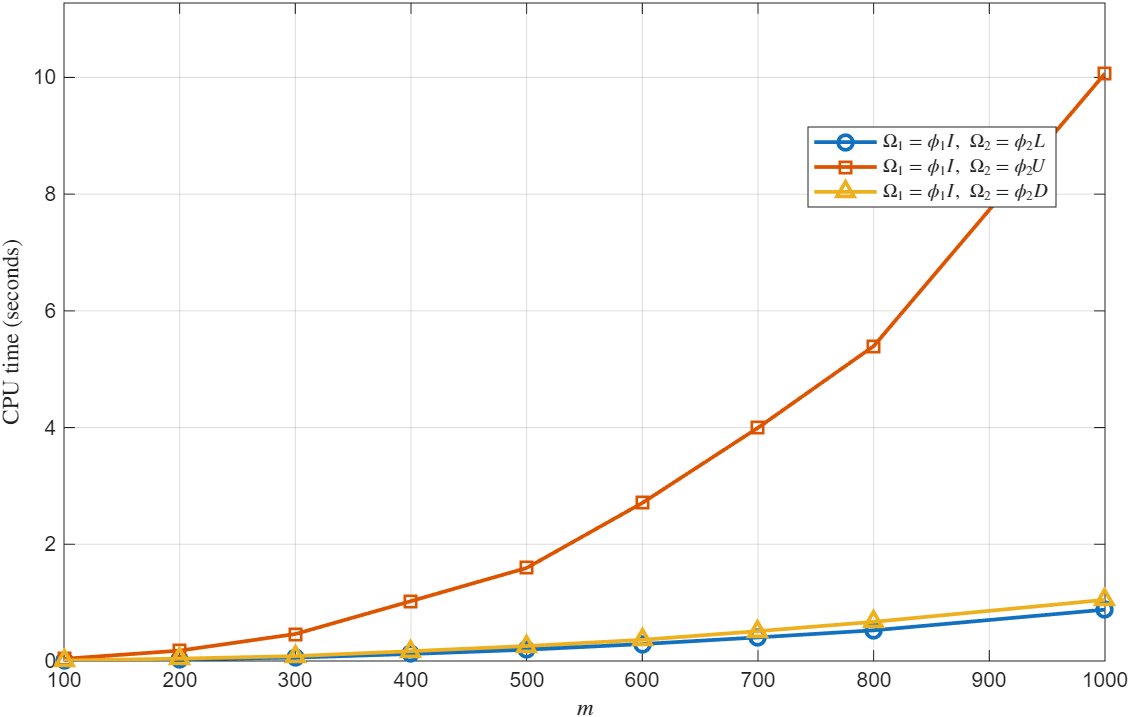}
	\end{tabular}
\caption{Sensitivity results for Example~\ref{Ex2}.}
	\label{fig:sensitivity-ex2}
\end{figure}

\begin{figure}[h]
	\centering
	\setlength{\tabcolsep}{1pt}
	\begin{tabular}{@{}cc@{}}
		\includegraphics[width=0.50\textwidth]{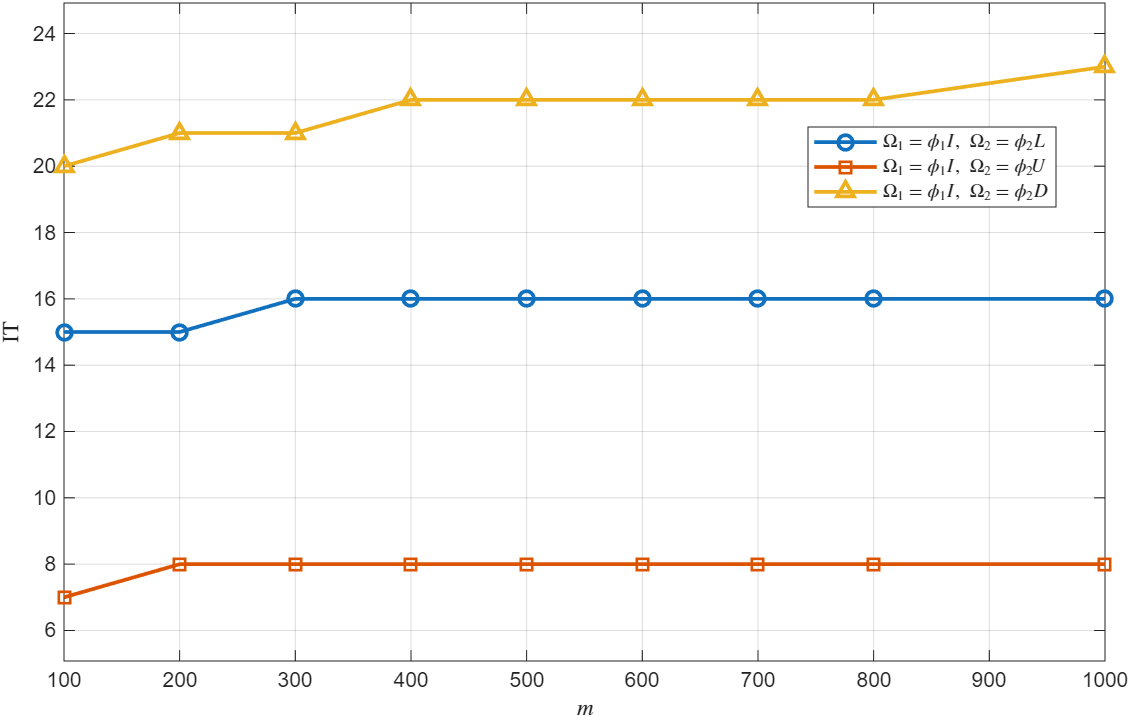}
		&
		\includegraphics[width=0.50\textwidth]{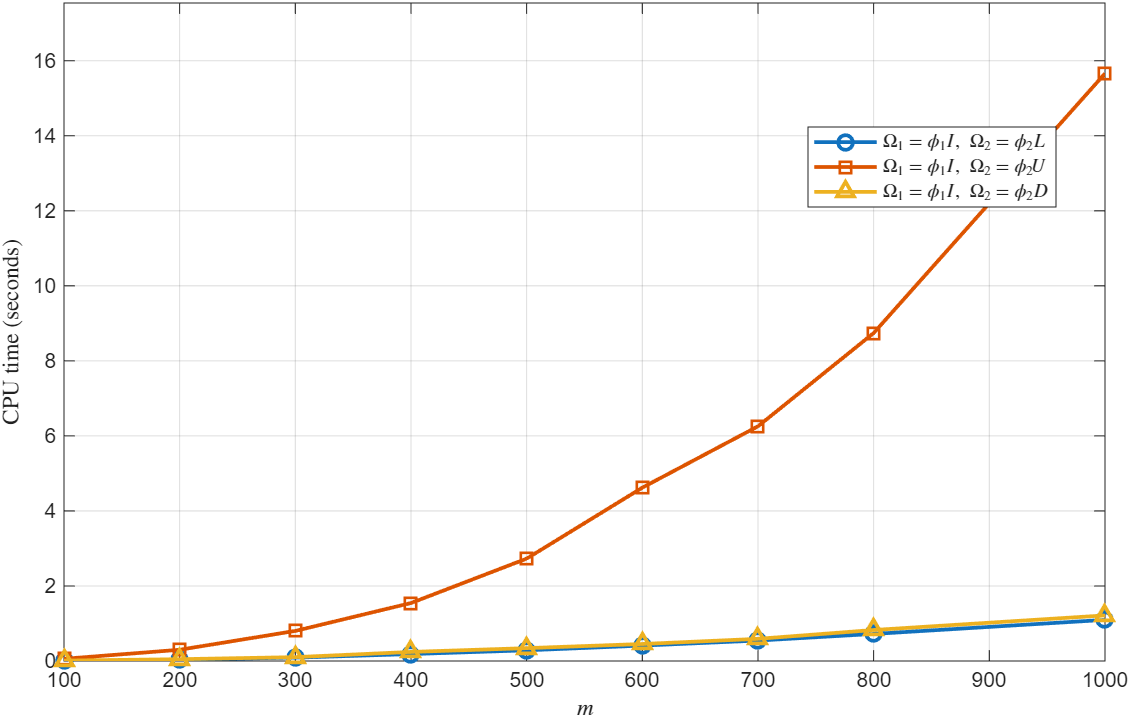}
	\end{tabular}
	\caption{Sensitivity results for Example~\ref{Ex3}.}
	\label{fig:sensitivity-ex3}
\end{figure}

\begin{figure}[h]
	\centering
	\setlength{\tabcolsep}{1pt}
	\begin{tabular}{@{}cc@{}}
		\includegraphics[width=0.50\textwidth]{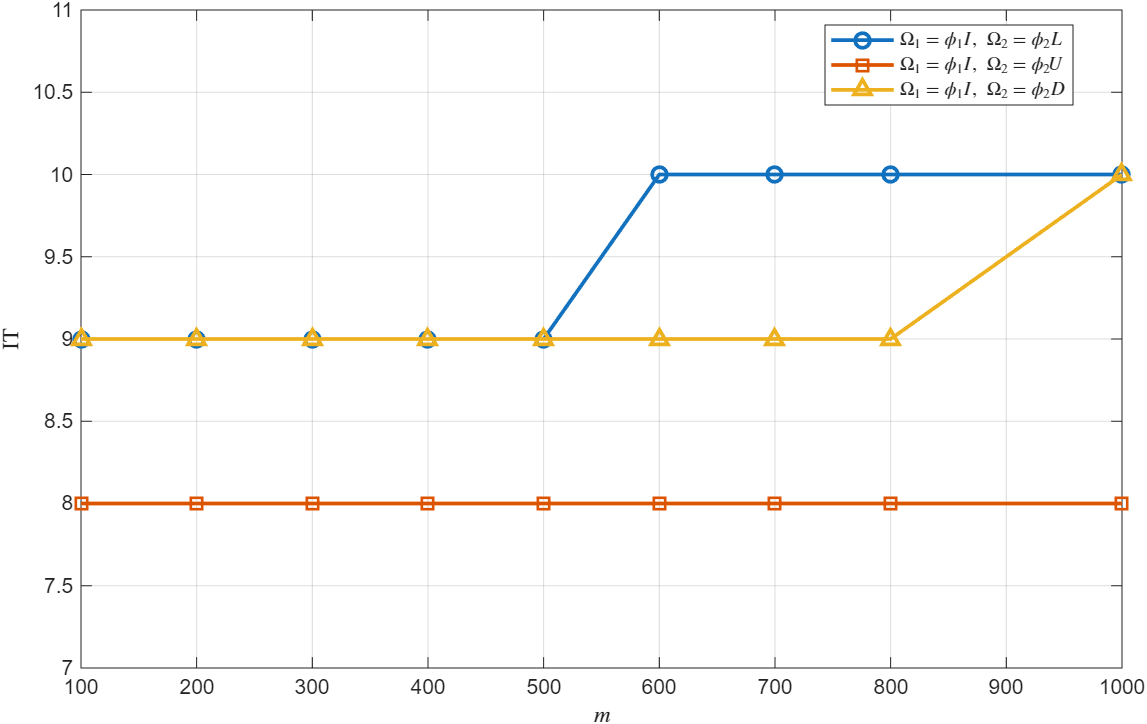}
		&
		\includegraphics[width=0.50\textwidth]{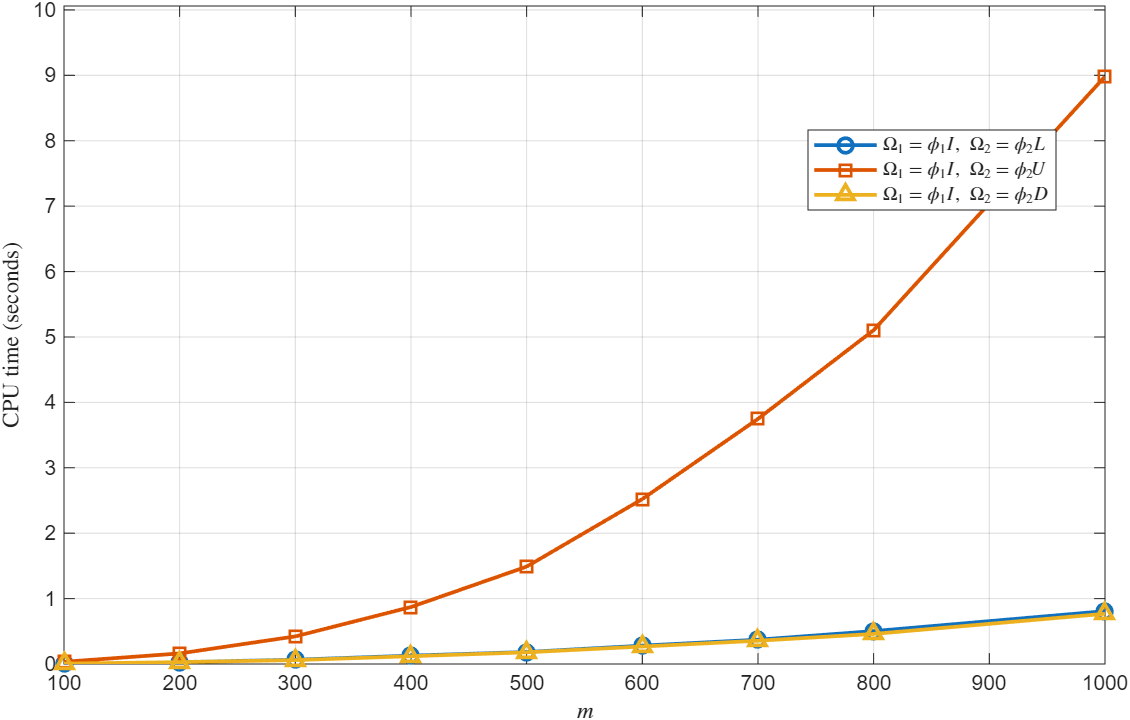}
	\end{tabular}
\caption{Sensitivity results for Example~\ref{Ex4}.}
	\label{fig:sensitivity-ex4}
\end{figure}

Figures~\ref{fig:sensitivity-ex1}--\ref{fig:sensitivity-ex4} show the
corresponding iteration counts and CPU times for all four examples with \(\mu=4\) as \(m\) increases. The results show that the choice
\(\Omega_1=\phi_1 I,\ \Omega_2=\phi_2 U\) often gives the smallest number of
iterations. However, this choice also leads to a much larger CPU time for large
values of \(m\). This happens because the upper triangular choice changes the left coefficient
matrix in the linear system solved at each iteration and makes the sparse
factorization more expensive.

On the other hand, the choice
\(\Omega_1=\phi_1 I,\ \Omega_2=\phi_2 L\), which is used in the proposed
Gauss-Seidel type method M4, gives a good balance between iteration count and CPU
time. The iteration count remains stable as \(m\) increases, and the CPU time
stays much smaller than that obtained with the upper triangular choice. The
diagonal choice \(\Omega_2=\phi_2 D\) is also stable, but in most cases it needs
more iterations.

Therefore, the lower triangular choice
\(\Omega_1=\phi_1 I,\ \Omega_2=\phi_2 L\) is used in the numerical experiments. It preserves the sparse triangular structure of the Gauss-Seidel
splitting and gives reliable performance for increasing problem sizes.

\subsection{Comparison between MAOR and RSMGS methods}
\label{ss:maor-rsmgs}

We now compare the proposed two-parameter RSMGS method with the two-parameter MAOR method obtained from Bai's MMS framework \cite{Bai2010}.

For the AOR-type splitting
\begin{equation*}
	M_{\alpha,\beta}
	=
	\frac{1}{\alpha}(D-\beta L),
	\qquad
	N_{\alpha,\beta}
	=
	\frac{1}{\alpha}
	\left[(1-\alpha)D+(\alpha-\beta)L+\alpha U\right],
\end{equation*}
the MAOR method corresponding to Bai's MMS framework can be written as
\begin{equation*}
	(D+\Omega-\beta L)x^{(k+1)}
	=
	\left((1-\alpha)D+(\alpha-\beta)L+\alpha U\right)x^{(k)}
	+
	(\Omega-\alpha A)|x^{(k)}|
	-
	2\alpha q,
\end{equation*}
where
$
	\Omega=\frac{1}{2\alpha}D.
$
Equivalently,
\begin{equation}
	\begin{aligned}
		\left(\left(1+\frac{1}{2\alpha}\right)D-\beta L\right)x^{(k+1)}
		={}&
		\left((1-\alpha)D+(\alpha-\beta)L+\alpha U\right)x^{(k)}  \\
		&+
		\left(\frac{1}{2\alpha}D-\alpha A\right)|x^{(k)}|
		-
		2\alpha q .
	\end{aligned}
	\label{eq:maor-comparison}
\end{equation}

On the other hand, the proposed RSMGS method uses the Gauss-Seidel splitting
\(M=D-L\), \(N=U\), together with the shifted matrices
\(\Omega_1=\phi_1I\), \(\Omega_2=\phi_2L\), \(\gamma=2\), and
\(\Omega_4=\frac12D\). Hence its iteration is
\begin{equation}
	\left(2D+\phi_1I-(1+\phi_2)L\right)x^{(k+1)}
	=
	\left(U+\phi_1I-\phi_2L\right)x^{(k)}
	+
	(D-A)|x^{(k)}|
	-
	2q.
	\label{eq:rsmgs-comparison}
\end{equation}

\begin{sidewaystable}[htbp]
	\centering
	\caption{Comparison between MAOR and RSMGS methods for \(\mu=4\) with
		\(\Omega_1=\phi_1I\) and \(\Omega_2=\phi_2L\)}
	\label{tab:maor-rsmgs-mu4}
	\scriptsize
	\renewcommand{\arraystretch}{1.08}
	\setlength{\tabcolsep}{3pt}
	\begin{tabular*}{\textwidth}{@{\extracolsep{\fill}}llclccccc@{}}
		\toprule
		Example & Method & Parameters &  & \multicolumn{5}{c}{$m$} \\
		\cmidrule(lr){5-9}
		& & & & 100 & 400 & 600 & 800 & 1000 \\
		\midrule
		
		\multirow{6}{*}{Example~\ref{Ex1}}
		& \multirow{3}{*}{MAOR}
		& \multirow{3}{*}{$(\alpha,\beta)=(0.80,2.50)$}
		& IT  & 12 & 12 & 13 & 13 & 13 \\
		& & & CPU & 0.0144 & 0.0988 & 0.2362 & 0.3840 & 0.6041 \\
		& & & RES & 4.6599e-06 & 8.4159e-06 & 2.6675e-06 & 3.2448e-06 & 3.8472e-06 \\
		& \multirow{3}{*}{RSMGS}
		& \multirow{3}{*}{$(\phi_1,\phi_2)=(-1.25,1.50)$}
		& IT  & 11 & 12 & 12 & 12 & 12 \\
		& & & CPU & 0.0204 & 0.1066 & 0.2168 & 0.3869 & 0.5809 \\
		& & & RES & 7.0378e-06 & 2.9705e-06 & 3.7634e-06 & 4.5938e-06 & 5.4445e-06 \\
		
		\addlinespace
		\multirow{6}{*}{Example~\ref{Ex2}}
		& \multirow{3}{*}{MAOR}
		& \multirow{3}{*}{$(\alpha,\beta)=(0.85,1.95)$}
		& IT  & 8 & 8 & 8 & 8 & 8 \\
		& & & CPU & 0.0042 & 0.0705 & 0.1407 & 0.2312 & 0.3809 \\
		& & & RES & 2.5411e-06 & 4.6393e-06 & 5.8208e-06 & 6.9348e-06 & 8.0095e-06 \\
		& \multirow{3}{*}{RSMGS}
		& \multirow{3}{*}{$(\phi_1,\phi_2)=(-0.75,1.25)$}
		& IT  & 8 & 8 & 8 & 8 & 8 \\
		& & & CPU & 0.0040 & 0.0626 & 0.1333 & 0.2387 & 0.4193 \\
		& & & RES & 1.2616e-06 & 1.9999e-06 & 2.3760e-06 & 2.7062e-06 & 3.0057e-06 \\
		
		\addlinespace
		\multirow{6}{*}{Example~\ref{Ex3}}
		& \multirow{3}{*}{MAOR}
		& \multirow{3}{*}{$(\alpha,\beta)=(0.88,2.54)$}
		& IT  & 15 & 16 & 16 & 17 & 17 \\
		& & & CPU & 0.0065 & 0.1310 & 0.2623 & 0.5020 & 0.7092 \\
		& & & RES & 6.3861e-06 & 5.8441e-06 & 8.5515e-06 & 3.4160e-06 & 4.2323e-06 \\
		& \multirow{3}{*}{RSMGS}
		& \multirow{3}{*}{$(\phi_1,\phi_2)=(-2,1.75)$}
		& IT  & 15 & 16 & 16 & 16 & 16 \\
		& & & CPU & 0.0097 & 0.1331 & 0.2632 & 0.4826 & 0.6720 \\
		& & & RES & 6.8608e-06 & 4.6944e-06 & 6.3721e-06 & 8.1400e-06 & 9.9501e-06 \\
		
		\addlinespace
		\multirow{6}{*}{Example~\ref{Ex4}}
		& \multirow{3}{*}{MAOR}
		& \multirow{3}{*}{$(\alpha,\beta)=(0.85,1.70)$}
		& IT  & 11 & 12 & 12 & 12 & 12 \\
		& & & CPU & 0.0052 & 0.1043 & 0.2017 & 0.3534 & 0.5583 \\
		& & & RES & 6.4676e-06 & 3.5363e-06 & 4.3339e-06 & 5.0068e-06 & 5.6001e-06 \\
		& \multirow{3}{*}{RSMGS}
		& \multirow{3}{*}{$(\phi_1,\phi_2)=(-0.75,0.25)$}
		& IT  & 9 & 9 & 10 & 10 & 10 \\
		& & & CPU & 0.0047 & 0.0806 & 0.1670 & 0.2850 & 0.4392 \\
		& & & RES & 4.2427e-06 & 8.4608e-06 & 1.9372e-06 & 2.2360e-06 & 2.4994e-06 \\
		
		\botrule
	\end{tabular*}
\end{sidewaystable}

Equations~\eqref{eq:maor-comparison} and \eqref{eq:rsmgs-comparison} show that
the two methods modify the iteration in different ways. In MAOR, the parameters
\(\alpha\) and \(\beta\) enter through the AOR splitting and also scale the
modulus term \((\Omega-\alpha A)|x^{(k)}|\) and the vector \(q\). In RSMGS, the
parameters \(\phi_1\) and \(\phi_2\) shift the Gauss-Seidel splitting itself:
they appear simultaneously in the coefficient matrix and in the linear part of
the right-hand side, while the modulus term remains \((D-A)|x^{(k)}|\). Thus,
RSMGS is not identical to MAOR, except in the unshifted case
\(\phi_1=\phi_2=0\) together with the corresponding Gauss-Seidel choice in the
MMS framework.

For a fair comparison, the best parameters of both methods are selected
experimentally at the fixed problem size \(m=100\) for each test problem. The
selected values of \((\alpha,\beta)\) for MAOR and \((\phi_1,\phi_2)\) for
RSMGS are then kept fixed for all larger values of \(m\).

Table~\ref{tab:maor-rsmgs-mu4} compares the two-parameter MAOR method with the
proposed two-parameter RSMGS method. For Examples~\ref{Ex1}--\ref{Ex3}, the two
methods show comparable performance, which is expected because both methods use
two adjustable parameters. However, the improvement obtained by RSMGS is more
pronounced for Example~\ref{Ex4}, which comes from a quasi-variational
inequality model related to a continuous optimal control problem. This indicates
that shifting the Gauss-Seidel splitting through \((\phi_1,\phi_2)\) can be
more effective than only relaxing the AOR-type splitting through
\((\alpha,\beta)\) for this class of problems. Hence, the proposed RSMGS method
is not merely a reformulation of MAOR; it changes the splitting mechanism and
can lead to different computational behavior.

\subsection{Results and Discussion}

Tables~\ref{tab:ex1}--\ref{tab:ex4} report the numerical results for the four
test problems with \(\mu=4\) and \(\mu=6\). The comparison is made in terms of
IT, CPU, and RES for \(m=100,400,600,800,1000\). The proposed method is denoted
by RSMGS, with the selected pair \((\phi_1,\phi_2)\) shown in each table.

For Example~\ref{Ex1}, RSMGS gives the smallest iteration counts among the
tested methods for both \(\mu=4\) and \(\mu=6\). Compared with MGS, MSOR,
NMGS, NPGS, and NPSOR, the reduction in IT is substantial, and the CPU time
remains competitive for all tested problem sizes. The MB-DS$_L$ method also
has relatively small iteration counts, but its CPU time grows rapidly as \(m\)
increases. RGTMSOR performs well, but RSMGS gives fewer iterations and smaller
CPU time for the largest problem sizes.

For the non-symmetric problem in Example~\ref{Ex2}, RSMGS again shows strong
performance. It gives the lowest iteration counts for both values of \(\mu\),
and its CPU time is among the smallest across all tested methods. The residuals
remain of order \(10^{-6}\), confirming that the observed improvement is not
obtained at the cost of accuracy. The same parameter pair
\((\phi_1,\phi_2)=(-0.75,1.25)\) is effective for both \(\mu=4\) and
\(\mu=6\).

Table~\ref{tab:ex3} shows that RSMGS is also effective for the sparse upper
block problem. For \(\mu=4\) and \(\mu=6\), it requires fewer iterations than
the other modulus-based and projected methods. Its CPU time remains low and
grows moderately with \(m\). This indicates that the shifted Gauss-Seidel
splitting remains stable even for the nonsymmetric sparse structure in
Example~\ref{Ex3}.

For the quasi-variational inequality problem in Example~\ref{Ex4}, the
advantage of RSMGS is particularly clear. The projected methods NPGS and NPSOR
require many more iterations, while MB-DS$_L$ has small iteration counts but
significantly larger CPU times for large \(m\). The proposed RSMGS method gives
low iteration counts and consistently small CPU times for both \(\mu=4\) and
\(\mu=6\). This confirms that the shifted splitting is effective not only for
standard block tridiagonal test problems, but also for the LCP arising from a
quasi-variational inequality model.

Overall, the numerical results show that RSMGS is efficient and robust for the
tested large-scale sparse LCPs. Across all four examples, it gives either the
best or a highly competitive performance in terms of IT and CPU, while
maintaining residuals of order \(10^{-6}\). The iteration counts remain nearly
stable as \(m\) increases, which indicates that the proposed shifted
Gauss-Seidel splitting is suitable for large-scale sparse problems.Together with the sensitivity analysis presented in
Figures~\ref{fig:sensitivity-ex1}--\ref{fig:sensitivity-ex4} and the MAOR
comparison reported in Table~\ref{tab:maor-rsmgs-mu4}, these results support
the effectiveness of the shifted-splitting parameters \((\phi_1,\phi_2)\).

\section{Conclusions and Future Work}\label{sec6}

In this paper, a relaxed shifted matrix-splitting modulus-based iteration
framework has been proposed for large-scale sparse linear complementarity
problems. Two shift matrices, \(\Omega_1\) and \(\Omega_2\), have been
introduced into the splitting without changing the original coefficient matrix
of the LCP. The resulting Gauss-Seidel-type scheme, denoted by RSMGS, has been
used in the numerical experiments.

Convergence results have been established for \(P\)-matrices and
\(H_+\)-matrices. For structured sparse matrices, practical AOR-type sufficient
conditions and explicit admissible intervals for the shift parameters
\((\phi_1,\phi_2)\) have also been derived. These theoretical intervals provide
guidance for the numerical parameter search, and the selected parameters have
been verified to satisfy the sufficient convergence condition for the tested
problems.

The numerical results show that RSMGS is efficient and robust for the tested
large-scale sparse problems in terms of iteration count, residual, and CPU
time. The sensitivity analysis supports the use of
\(\Omega_1=\phi_1I,\ \Omega_2=\phi_2L\), which preserves the sparse
Gauss-Seidel structure and gives a good balance between convergence speed and
computational cost. The direct comparison with MAOR also shows that the
proposed shifted splitting leads to a genuinely different iteration mechanism
and can improve the computational performance.

The selection of efficient values of \(\phi_1\) and \(\phi_2\) remains an
important practical issue. Although admissible intervals and a grid-search
procedure have been provided in this work, a fully analytical strategy for
determining optimal or nearly optimal shift parameters is still open. Future
work will focus on this question and on extending the shifted-splitting idea to
implicit, horizontal, and vertical linear complementarity problems.

\section*{Acknowledgements}
The first author acknowledges the support of the Ministry of Education, Government of India, through the GATE fellowship (Registration No. MA22S25023148).

\bibliography{sn-bibliography}


\begin{thebibliography}{44}
\ifx \bisbn   \undefined \def \bisbn  #1{ISBN #1}\fi
\ifx \binits  \undefined \def \binits#1{#1}\fi
\ifx \bauthor  \undefined \def \bauthor#1{#1}\fi
\ifx \batitle  \undefined \def \batitle#1{#1}\fi
\ifx \bjtitle  \undefined \def \bjtitle#1{#1}\fi
\ifx \bvolume  \undefined \def \bvolume#1{\textbf{#1}}\fi
\ifx \byear  \undefined \def \byear#1{#1}\fi
\ifx \bissue  \undefined \def \bissue#1{#1}\fi
\ifx \bfpage  \undefined \def \bfpage#1{#1}\fi
\ifx \blpage  \undefined \def \blpage #1{#1}\fi
\ifx \burl  \undefined \def \burl#1{\textsf{#1}}\fi
\ifx \doiurl  \undefined \def \doiurl#1{\url{https://doi.org/#1}}\fi
\ifx \betal  \undefined \def \betal{\textit{et al.}}\fi
\ifx \binstitute  \undefined \def \binstitute#1{#1}\fi
\ifx \binstitutionaled  \undefined \def \binstitutionaled#1{#1}\fi
\ifx \bctitle  \undefined \def \bctitle#1{#1}\fi
\ifx \beditor  \undefined \def \beditor#1{#1}\fi
\ifx \bpublisher  \undefined \def \bpublisher#1{#1}\fi
\ifx \bbtitle  \undefined \def \bbtitle#1{#1}\fi
\ifx \bedition  \undefined \def \bedition#1{#1}\fi
\ifx \bseriesno  \undefined \def \bseriesno#1{#1}\fi
\ifx \blocation  \undefined \def \blocation#1{#1}\fi
\ifx \bsertitle  \undefined \def \bsertitle#1{#1}\fi
\ifx \bsnm \undefined \def \bsnm#1{#1}\fi
\ifx \bsuffix \undefined \def \bsuffix#1{#1}\fi
\ifx \bparticle \undefined \def \bparticle#1{#1}\fi
\ifx \barticle \undefined \def \barticle#1{#1}\fi
\bibcommenthead
\ifx \bconfdate \undefined \def \bconfdate #1{#1}\fi
\ifx \botherref \undefined \def \botherref #1{#1}\fi
\ifx \url \undefined \def \url#1{\textsf{#1}}\fi
\ifx \bchapter \undefined \def \bchapter#1{#1}\fi
\ifx \bbook \undefined \def \bbook#1{#1}\fi
\ifx \bcomment \undefined \def \bcomment#1{#1}\fi
\ifx \oauthor \undefined \def \oauthor#1{#1}\fi
\ifx \citeauthoryear \undefined \def \citeauthoryear#1{#1}\fi
\ifx \endbibitem  \undefined \def \endbibitem {}\fi
\ifx \bconflocation  \undefined \def \bconflocation#1{#1}\fi
\ifx \arxivurl  \undefined \def \arxivurl#1{\textsf{#1}}\fi
\csname PreBibitemsHook\endcsname

\bibitem[\protect\citeauthoryear{Murty}{1988}]{Murty1988}
\begin{bbook}
\bauthor{\bsnm{Murty}, \binits{K.G.}}:
\bbtitle{Linear Complementarity, Linear and Nonlinear Programming}.
\bpublisher{Heldermann Verlag},
\blocation{Berlin}
(\byear{1988})
\end{bbook}
\endbibitem

\bibitem[\protect\citeauthoryear{Cottle et~al.}{2009}]{CottlePangStone2009}
\begin{bbook}
\bauthor{\bsnm{Cottle}, \binits{R.W.}},
\bauthor{\bsnm{Pang}, \binits{J.-S.}},
\bauthor{\bsnm{Stone}, \binits{R.E.}}:
\bbtitle{The Linear Complementarity Problem},
\bedition{Classics in applied mathematics} edn.
\bpublisher{Society for Industrial and Applied Mathematics},
\blocation{Philadelphia}
(\byear{2009}).
\doiurl{10.1137/1.9780898719000}
\end{bbook}
\endbibitem

\bibitem[\protect\citeauthoryear{Shi et~al.}{2016}]{ShiYangHuang2016}
\begin{barticle}
\bauthor{\bsnm{Shi}, \binits{X.-J.}},
\bauthor{\bsnm{Yang}, \binits{L.}},
\bauthor{\bsnm{Huang}, \binits{Z.-H.}}:
\batitle{A fixed point method for the linear complementarity problem arising
  from american option pricing}.
\bjtitle{Acta Mathematicae Applicatae Sinica, English Series}
\bvolume{32},
\bfpage{921}--\blpage{932}
(\byear{2016})
\doiurl{10.1007/s10255-016-0613-6}
\end{barticle}
\endbibitem

\bibitem[\protect\citeauthoryear{Lemke}{1965}]{Lemke1965}
\begin{barticle}
\bauthor{\bsnm{Lemke}, \binits{C.E.}}:
\batitle{Bimatrix equilibrium points and mathematical programming}.
\bjtitle{Management Science}
\bvolume{11}(\bissue{7}),
\bfpage{681}--\blpage{689}
(\byear{1965})
\doiurl{10.1287/mnsc.11.7.681}
\end{barticle}
\endbibitem

\bibitem[\protect\citeauthoryear{Cottle and Dantzig}{1968}]{CottleDantzig1968}
\begin{barticle}
\bauthor{\bsnm{Cottle}, \binits{R.W.}},
\bauthor{\bsnm{Dantzig}, \binits{G.B.}}:
\batitle{Complementary pivot theory of mathematical programming}.
\bjtitle{Linear Algebra and its Applications}
\bvolume{1}(\bissue{1}),
\bfpage{103}--\blpage{125}
(\byear{1968})
\doiurl{10.1016/0024-3795(68)90052-9}
\end{barticle}
\endbibitem

\bibitem[\protect\citeauthoryear{Kojima et~al.}{1992}]{KojimaMegiddoYe1992}
\begin{barticle}
\bauthor{\bsnm{Kojima}, \binits{M.}},
\bauthor{\bsnm{Megiddo}, \binits{N.}},
\bauthor{\bsnm{Ye}, \binits{Y.}}:
\batitle{An interior point potential reduction algorithm for the linear
  complementarity problem}.
\bjtitle{Mathematical Programming}
\bvolume{54},
\bfpage{267}--\blpage{279}
(\byear{1992})
\doiurl{10.1007/BF01586054}
\end{barticle}
\endbibitem

\bibitem[\protect\citeauthoryear{Cryer}{1971}]{Cryer1971}
\begin{barticle}
\bauthor{\bsnm{Cryer}, \binits{C.W.}}:
\batitle{The solution of a quadratic programming problem using systematic
  overrelaxation}.
\bjtitle{SIAM Journal on Control}
\bvolume{9}(\bissue{3}),
\bfpage{385}--\blpage{392}
(\byear{1971})
\doiurl{10.1137/0309028}
\end{barticle}
\endbibitem

\bibitem[\protect\citeauthoryear{Mangasarian}{1977}]{Mangasarian1977}
\begin{barticle}
\bauthor{\bsnm{Mangasarian}, \binits{O.L.}}:
\batitle{Solution of symmetric linear complementarity problems by iterative
  methods}.
\bjtitle{Journal of Optimization Theory and Applications}
\bvolume{22}(\bissue{4}),
\bfpage{465}--\blpage{485}
(\byear{1977})
\doiurl{10.1007/BF01268170}
\end{barticle}
\endbibitem

\bibitem[\protect\citeauthoryear{Ahn}{1981}]{Ahn1981}
\begin{barticle}
\bauthor{\bsnm{Ahn}, \binits{B.H.}}:
\batitle{Solution of nonsymmetric linear complementarity problems by iterative
  methods}.
\bjtitle{Journal of Optimization Theory and Applications}
\bvolume{33}(\bissue{2}),
\bfpage{175}--\blpage{185}
(\byear{1981})
\doiurl{10.1007/BF00935545}
\end{barticle}
\endbibitem

\bibitem[\protect\citeauthoryear{Pang}{1984}]{Pang1984}
\begin{barticle}
\bauthor{\bsnm{Pang}, \binits{J.S.}}:
\batitle{Necessary and sufficient conditions for the convergence of iterative
  methods for the linear complementarity problem}.
\bjtitle{Journal of Optimization Theory and Applications}
\bvolume{42}(\bissue{1}),
\bfpage{1}--\blpage{17}
(\byear{1984})
\doiurl{10.1007/BF00934130}
\end{barticle}
\endbibitem

\bibitem[\protect\citeauthoryear{Yuan and Song}{2003}]{YuanSong2003}
\begin{barticle}
\bauthor{\bsnm{Yuan}, \binits{D.}},
\bauthor{\bsnm{Song}, \binits{Y.}}:
\batitle{Modified {AOR} methods for linear complementarity problem}.
\bjtitle{Applied Mathematics and Computation}
\bvolume{140}(\bissue{1}),
\bfpage{53}--\blpage{67}
(\byear{2003})
\doiurl{10.1016/S0096-3003(02)00194-7}
\end{barticle}
\endbibitem

\bibitem[\protect\citeauthoryear{Li and Dai}{2007}]{LiDai2007}
\begin{barticle}
\bauthor{\bsnm{Li}, \binits{Y.}},
\bauthor{\bsnm{Dai}, \binits{P.}}:
\batitle{Generalized {AOR} methods for linear complementarity problem}.
\bjtitle{Applied Mathematics and Computation}
\bvolume{188}(\bissue{1}),
\bfpage{7}--\blpage{18}
(\byear{2007})
\doiurl{10.1016/j.amc.2006.09.067}
\end{barticle}
\endbibitem

\bibitem[\protect\citeauthoryear{Dehghan and
  Hajarian}{2009}]{DehghanHajarian2009}
\begin{barticle}
\bauthor{\bsnm{Dehghan}, \binits{M.}},
\bauthor{\bsnm{Hajarian}, \binits{M.}}:
\batitle{Convergence of {SSOR} methods for linear complementarity problems}.
\bjtitle{Operations Research Letters}
\bvolume{37}(\bissue{3}),
\bfpage{219}--\blpage{223}
(\byear{2009})
\doiurl{10.1016/j.orl.2009.01.013}
\end{barticle}
\endbibitem

\bibitem[\protect\citeauthoryear{van Bokhoven}{1981}]{vanBokhoven1981}
\begin{botherref}
\oauthor{\bsnm{Bokhoven}, \binits{W.M.G.}}:
Piecewise-linear modelling and analysis.
PhD thesis,
Technische Hogeschool Eindhoven
(1981).
\doiurl{10.6100/IR118197}
\end{botherref}
\endbibitem

\bibitem[\protect\citeauthoryear{Dong and Jiang}{2009}]{DongJiang2009}
\begin{barticle}
\bauthor{\bsnm{Dong}, \binits{J.L.}},
\bauthor{\bsnm{Jiang}, \binits{M.Q.}}:
\batitle{A modified modulus method for symmetric positive-definite linear
  complementarity problems}.
\bjtitle{Numerical Linear Algebra with Applications}
\bvolume{16}(\bissue{2}),
\bfpage{129}--\blpage{143}
(\byear{2009})
\doiurl{10.1002/nla.609}
\end{barticle}
\endbibitem

\bibitem[\protect\citeauthoryear{Hadjidimos and
  Tzoumas}{2009}]{HadjidimosTzoumas2009}
\begin{barticle}
\bauthor{\bsnm{Hadjidimos}, \binits{A.}},
\bauthor{\bsnm{Tzoumas}, \binits{M.}}:
\batitle{Nonstationary extrapolated modulus algorithms for the solution of the
  linear complementarity problem}.
\bjtitle{Linear Algebra and its Applications}
\bvolume{431}(\bissue{1-2}),
\bfpage{197}--\blpage{210}
(\byear{2009})
\doiurl{10.1016/j.laa.2009.02.024}
\end{barticle}
\endbibitem

\bibitem[\protect\citeauthoryear{Bai}{2010}]{Bai2010}
\begin{barticle}
\bauthor{\bsnm{Bai}, \binits{Z.-Z.}}:
\batitle{Modulus-based matrix splitting iteration methods for linear
  complementarity problems}.
\bjtitle{Numerical Linear Algebra with Applications}
\bvolume{17}(\bissue{6}),
\bfpage{917}--\blpage{933}
(\byear{2010})
\doiurl{10.1002/nla.680}
\end{barticle}
\endbibitem

\bibitem[\protect\citeauthoryear{Zhang}{2011}]{Zhang2011}
\begin{barticle}
\bauthor{\bsnm{Zhang}, \binits{L.-L.}}:
\batitle{Two-step modulus-based matrix splitting iteration method for linear
  complementarity problems}.
\bjtitle{Numerical Algorithms}
\bvolume{57}(\bissue{1}),
\bfpage{83}--\blpage{99}
(\byear{2011})
\doiurl{10.1007/s11075-010-9416-7}
\end{barticle}
\endbibitem

\bibitem[\protect\citeauthoryear{Zheng and Yin}{2013}]{ZhengYin2013}
\begin{barticle}
\bauthor{\bsnm{Zheng}, \binits{N.}},
\bauthor{\bsnm{Yin}, \binits{J.F.}}:
\batitle{Accelerated modulus-based matrix splitting iteration methods for
  linear complementarity problem}.
\bjtitle{Numerical Algorithms}
\bvolume{64}(\bissue{2}),
\bfpage{245}--\blpage{262}
(\byear{2013})
\doiurl{10.1007/s11075-012-9664-9}
\end{barticle}
\endbibitem

\bibitem[\protect\citeauthoryear{Bai and Zhang}{2013a}]{BaiZhang2013a}
\begin{barticle}
\bauthor{\bsnm{Bai}, \binits{Z.-Z.}},
\bauthor{\bsnm{Zhang}, \binits{L.-L.}}:
\batitle{Modulus-based synchronous multisplitting iteration methods for linear
  complementarity problems}.
\bjtitle{Numerical Linear Algebra with Applications}
\bvolume{20}(\bissue{3}),
\bfpage{425}--\blpage{439}
(\byear{2013})
\doiurl{10.1002/nla.1835}
\end{barticle}
\endbibitem

\bibitem[\protect\citeauthoryear{Bai and Zhang}{2013b}]{BaiZhang2013b}
\begin{barticle}
\bauthor{\bsnm{Bai}, \binits{Z.-Z.}},
\bauthor{\bsnm{Zhang}, \binits{L.-L.}}:
\batitle{Modulus-based synchronous two-stage multisplitting iteration methods
  for linear complementarity problems}.
\bjtitle{Numerical Algorithms}
\bvolume{62},
\bfpage{59}--\blpage{77}
(\byear{2013})
\doiurl{10.1007/s11075-012-9566-x}
\end{barticle}
\endbibitem

\bibitem[\protect\citeauthoryear{Li}{2013}]{Li2013}
\begin{barticle}
\bauthor{\bsnm{Li}, \binits{W.}}:
\batitle{A general modulus-based matrix splitting method for linear
  complementarity problems of {H}-matrices}.
\bjtitle{Applied Mathematics Letters}
\bvolume{26}(\bissue{12}),
\bfpage{1159}--\blpage{1164}
(\byear{2013})
\doiurl{10.1016/j.aml.2013.06.015}
\end{barticle}
\endbibitem

\bibitem[\protect\citeauthoryear{Xu}{2015}]{Xu2015}
\begin{barticle}
\bauthor{\bsnm{Xu}, \binits{W.-W.}}:
\batitle{Modified modulus-based matrix splitting iteration methods for linear
  complementarity problems}.
\bjtitle{Numerical Linear Algebra with Applications}
\bvolume{22}(\bissue{4}),
\bfpage{748}--\blpage{760}
(\byear{2015})
\doiurl{10.1002/nla.1985}
\end{barticle}
\endbibitem

\bibitem[\protect\citeauthoryear{Li and Zheng}{2016}]{LiZheng2016}
\begin{barticle}
\bauthor{\bsnm{Li}, \binits{W.}},
\bauthor{\bsnm{Zheng}, \binits{H.}}:
\batitle{A preconditioned modulus-based iteration method for solving linear
  complementarity problems of {H}-matrices}.
\bjtitle{Linear and Multilinear Algebra}
\bvolume{64}(\bissue{7}),
\bfpage{1390}--\blpage{1403}
(\byear{2016})
\doiurl{10.1080/03081087.2015.1087457}
\end{barticle}
\endbibitem

\bibitem[\protect\citeauthoryear{Wu and Li}{2016}]{WuLi2016}
\begin{barticle}
\bauthor{\bsnm{Wu}, \binits{S.-L.}},
\bauthor{\bsnm{Li}, \binits{C.-X.}}:
\batitle{Two-sweep modulus-based matrix splitting iteration methods for linear
  complementarity problems}.
\bjtitle{Journal of Computational and Applied Mathematics}
\bvolume{302},
\bfpage{327}--\blpage{339}
(\byear{2016})
\doiurl{10.1016/j.cam.2016.02.011}
\end{barticle}
\endbibitem

\bibitem[\protect\citeauthoryear{Wu and Li}{2022}]{Wu2022}
\begin{barticle}
\bauthor{\bsnm{Wu}, \binits{S.-L.}},
\bauthor{\bsnm{Li}, \binits{C.-X.}}:
\batitle{A class of new modulus-based matrix splitting methods for linear
  complementarity problem}.
\bjtitle{Optimization Letters}
\bvolume{16},
\bfpage{1427}--\blpage{1443}
(\byear{2022})
\doiurl{10.1007/s11590-021-01781-6}
\end{barticle}
\endbibitem

\bibitem[\protect\citeauthoryear{Kumar
  et~al.}{2023}]{KumarDeepmalaDuttaDas2023}
\begin{barticle}
\bauthor{\bsnm{Kumar}, \binits{B.}},
\bauthor{\bsnm{Deepmala}},
\bauthor{\bsnm{Dutta}, \binits{A.}},
\bauthor{\bsnm{Das}, \binits{A.K.}}:
\batitle{More on matrix splitting modulus-based iterative methods for solving
  linear complementarity problem}.
\bjtitle{OPSEARCH}
\bvolume{60},
\bfpage{1003}--\blpage{1020}
(\byear{2023})
\doiurl{10.1007/s12597-023-00634-3}
\end{barticle}
\endbibitem

\bibitem[\protect\citeauthoryear{Fang and Zhu}{2019}]{Fang2019}
\begin{barticle}
\bauthor{\bsnm{Fang}, \binits{X.-M.}},
\bauthor{\bsnm{Zhu}, \binits{Z.-W.}}:
\batitle{The modulus-based matrix double splitting iteration method for linear
  complementarity problems}.
\bjtitle{Computers \& Mathematics with Applications}
\bvolume{78}(\bissue{11}),
\bfpage{3633}--\blpage{3643}
(\byear{2019})
\doiurl{10.1016/j.camwa.2019.06.012}
\end{barticle}
\endbibitem

\bibitem[\protect\citeauthoryear{Zheng et~al.}{2017}]{ZhengLiVong2017}
\begin{barticle}
\bauthor{\bsnm{Zheng}, \binits{H.}},
\bauthor{\bsnm{Li}, \binits{W.}},
\bauthor{\bsnm{Vong}, \binits{S.}}:
\batitle{A relaxation modulus-based matrix splitting iteration method for
  solving linear complementarity problems}.
\bjtitle{Numerical Algorithms}
\bvolume{74},
\bfpage{137}--\blpage{152}
(\byear{2017})
\doiurl{10.1007/s11075-016-0142-7}
\end{barticle}
\endbibitem

\bibitem[\protect\citeauthoryear{Wen et~al.}{2018}]{WenZhengLiPeng2018}
\begin{barticle}
\bauthor{\bsnm{Wen}, \binits{B.-L.}},
\bauthor{\bsnm{Zheng}, \binits{H.}},
\bauthor{\bsnm{Li}, \binits{W.}},
\bauthor{\bsnm{Peng}, \binits{X.-F.}}:
\batitle{The relaxation modulus-based matrix splitting iteration method for
  solving linear complementarity problems of positive definite matrices}.
\bjtitle{Applied Mathematics and Computation}
\bvolume{321},
\bfpage{349}--\blpage{357}
(\byear{2018})
\doiurl{10.1016/j.amc.2017.10.064}
\end{barticle}
\endbibitem

\bibitem[\protect\citeauthoryear{Ren et~al.}{2019}]{RenWangTangWang2019}
\begin{barticle}
\bauthor{\bsnm{Ren}, \binits{H.}},
\bauthor{\bsnm{Wang}, \binits{X.}},
\bauthor{\bsnm{Tang}, \binits{X.-B.}},
\bauthor{\bsnm{Wang}, \binits{T.}}:
\batitle{A preconditioned general two-step modulus-based matrix splitting
  iteration method for linear complementarity problems of {H+}-matrices}.
\bjtitle{Numerical Algorithms}
\bvolume{82},
\bfpage{969}--\blpage{986}
(\byear{2019})
\doiurl{10.1007/s11075-018-0637-5}
\end{barticle}
\endbibitem

\bibitem[\protect\citeauthoryear{Huang and Cui}{2023}]{HuangCui2023}
\begin{barticle}
\bauthor{\bsnm{Huang}, \binits{Z.}},
\bauthor{\bsnm{Cui}, \binits{J.}}:
\batitle{The double-relaxation modulus-based matrix splitting iteration method
  for linear complementarity problems}.
\bjtitle{Journal of Computational and Applied Mathematics}
\bvolume{427},
\bfpage{115138}
(\byear{2023})
\doiurl{10.1016/j.cam.2023.115138}
\end{barticle}
\endbibitem

\bibitem[\protect\citeauthoryear{Liu et~al.}{2025}]{LiuWangHuang2025}
\begin{barticle}
\bauthor{\bsnm{Liu}, \binits{W.}},
\bauthor{\bsnm{Wang}, \binits{Y.}},
\bauthor{\bsnm{Huang}, \binits{Z.-H.}}:
\batitle{Improved {G}auss-{S}eidel type and {J}acobi type methods for linear
  complementarity problems}.
\bjtitle{Computational Optimization and Applications}
\bvolume{92},
\bfpage{683}--\blpage{708}
(\byear{2025})
\doiurl{10.1007/s10589-025-00714-8}
\end{barticle}
\endbibitem

\bibitem[\protect\citeauthoryear{Liu and Zhang}{2026}]{LiuZhang2026}
\begin{barticle}
\bauthor{\bsnm{Liu}, \binits{Y.}},
\bauthor{\bsnm{Zhang}, \binits{J.-J.}}:
\batitle{Acceleration of the modulus-based iteration methods for solving linear
  complementarity problems by momentum}.
\bjtitle{Computational and Applied Mathematics}
\bvolume{45}(\bissue{4}),
\bfpage{135}
(\byear{2026})
\doiurl{10.1007/s40314-025-03528-w}
\end{barticle}
\endbibitem

\bibitem[\protect\citeauthoryear{Cvetkovi{\'c} et~al.}{2014}]{Cvetkovic2014}
\begin{barticle}
\bauthor{\bsnm{Cvetkovi{\'c}}, \binits{L.}},
\bauthor{\bsnm{Hadjidimos}, \binits{A.}},
\bauthor{\bsnm{Kosti{\'c}}, \binits{V.}}:
\batitle{On the choice of parameters in {MAOR} type splitting methods for the
  linear complementarity problem}.
\bjtitle{Numerical Algorithms}
\bvolume{67},
\bfpage{793}--\blpage{806}
(\byear{2014})
\doiurl{10.1007/s11075-014-9824-1}
\end{barticle}
\endbibitem

\bibitem[\protect\citeauthoryear{Bai et~al.}{2006}]{BaiYinSu2006}
\begin{barticle}
\bauthor{\bsnm{Bai}, \binits{Z.-Z.}},
\bauthor{\bsnm{Yin}, \binits{J.-F.}},
\bauthor{\bsnm{Su}, \binits{Y.-F.}}:
\batitle{A shift-splitting preconditioner for non-hermitian positive definite
  matrices}.
\bjtitle{Journal of Computational Mathematics}
\bvolume{24}(\bissue{4}),
\bfpage{539}--\blpage{552}
(\byear{2006})
\end{barticle}
\endbibitem

\bibitem[\protect\citeauthoryear{Cao et~al.}{2014}]{CaoEtAl2014}
\begin{barticle}
\bauthor{\bsnm{Cao}, \binits{Y.}},
\bauthor{\bsnm{Du}, \binits{J.}},
\bauthor{\bsnm{Niu}, \binits{Q.}}:
\batitle{Shift-splitting preconditioners for saddle point problems}.
\bjtitle{Journal of Computational and Applied Mathematics}
\bvolume{272},
\bfpage{239}--\blpage{250}
(\byear{2014})
\doiurl{10.1016/j.cam.2014.05.017}
\end{barticle}
\endbibitem

\bibitem[\protect\citeauthoryear{Li and Ma}{2019}]{LiMa2019}
\begin{barticle}
\bauthor{\bsnm{Li}, \binits{C.-L.}},
\bauthor{\bsnm{Ma}, \binits{C.-F.}}:
\batitle{Efficient parameterized rotated shift-splitting preconditioner for a
  class of complex symmetric linear systems}.
\bjtitle{Numerical Algorithms}
\bvolume{80},
\bfpage{337}--\blpage{354}
(\byear{2019})
\doiurl{10.1007/s11075-018-0487-1}
\end{barticle}
\endbibitem

\bibitem[\protect\citeauthoryear{Berman and
  Plemmons}{1994}]{BermanPlemmons1994}
\begin{bbook}
\bauthor{\bsnm{Berman}, \binits{A.}},
\bauthor{\bsnm{Plemmons}, \binits{R.J.}}:
\bbtitle{Nonnegative Matrices in the Mathematical Sciences}.
\bpublisher{Society for Industrial and Applied Mathematics},
\blocation{Philadelphia}
(\byear{1994}).
\doiurl{10.1137/1.9781611971262}
\end{bbook}
\endbibitem

\bibitem[\protect\citeauthoryear{Varga}{2000}]{Varga2000}
\begin{bbook}
\bauthor{\bsnm{Varga}, \binits{R.S.}}:
\bbtitle{Matrix Iterative Analysis},
\bedition{2}nd edn.
\bpublisher{Springer},
\blocation{Berlin}
(\byear{2000}).
\doiurl{10.1007/978-3-642-05156-2}
\end{bbook}
\endbibitem

\bibitem[\protect\citeauthoryear{Horn and Johnson}{2013}]{Horn2013}
\begin{bbook}
\bauthor{\bsnm{Horn}, \binits{R.A.}},
\bauthor{\bsnm{Johnson}, \binits{C.R.}}:
\bbtitle{Matrix Analysis},
\bedition{2}nd edn.
\bpublisher{Cambridge University Press},
\blocation{Cambridge}
(\byear{2013})
\end{bbook}
\endbibitem

\bibitem[\protect\citeauthoryear{Das et~al.}{2025}]{Das2025}
\begin{barticle}
\bauthor{\bsnm{Das}, \binits{A.K.}},
\bauthor{\bsnm{Deepmala}},
\bauthor{\bsnm{Kumar}, \binits{B.}}:
\batitle{More on projected type iteration method and linear complementarity
  problem}.
\bjtitle{Proceedings of the Indian Academy of Sciences: Mathematical Sciences}
\bvolume{135},
\bfpage{35}
(\byear{2025})
\doiurl{10.1007/s12044-025-00844-3}
\end{barticle}
\endbibitem

\bibitem[\protect\citeauthoryear{Li et~al.}{2022}]{LiWangLiu2022}
\begin{barticle}
\bauthor{\bsnm{Li}, \binits{D.-K.}},
\bauthor{\bsnm{Wang}, \binits{L.}},
\bauthor{\bsnm{Liu}, \binits{Y.-Y.}}:
\batitle{A relaxation general two-sweep modulus-based matrix splitting
  iteration method for solving linear complementarity problems}.
\bjtitle{Journal of Computational and Applied Mathematics}
\bvolume{409},
\bfpage{114140}
(\byear{2022})
\doiurl{10.1016/j.cam.2022.114140}
\end{barticle}
\endbibitem

\bibitem[\protect\citeauthoryear{Isac}{2006}]{Isac2006}
\begin{bchapter}
\bauthor{\bsnm{Isac}, \binits{G.}}:
\bctitle{Leray-schauder type alternatives. existence theorems}.
In: \bbtitle{Leray-Schauder Type Alternatives, Complementarity Problems and
  Variational Inequalities},
pp. \bfpage{137}--\blpage{223}.
\bpublisher{Springer},
\blocation{Boston, MA}
(\byear{2006}).
\doiurl{10.1007/0-387-32900-5_5}
\end{bchapter}
\endbibitem

\end{thebibliography}

\end{document}